\documentclass[12pt]{amsart}

\usepackage[T1]{fontenc}
\usepackage{newtxtext,newtxmath}
\usepackage{microtype}
\usepackage{comment}
\usepackage{latexsym,amscd, graphicx, color, amsthm, bm, amsmath, cancel, enumitem,chessboard, skak,mathtools}  
\usepackage{tikz, tikz-3dplot}
\usepackage[margin=1in]{geometry}
\usepackage{hyperref}
\usepackage[all,cmtip]{xy}

\numberwithin{equation}{section}

\newtheorem{theorem}{Theorem}[section]
\newtheorem{proposition}[theorem]{Proposition}

\newtheorem{lemma}[theorem]{Lemma}
\newtheorem{conjecture}[theorem]{Conjecture}
\newtheorem{observation}[theorem]{Observation}

\newtheorem{problem}[theorem]{Problem}

\newtheorem{definition}[theorem]{Definition}

\theoremstyle{definition}

\newcommand{\Hilb}{{\mathrm{Hilb}}}

\newcommand{\bbbb}{{\mathfrak{b}}}

\newcommand{\symm}{{\mathfrak{S}}}

\newcommand{\Cone}{{\mathrm{Cone}}}

\newcommand{\one}{{\mathbf{1}}}

\newcommand{\AAA}{{\mathcal{A}}}

\newcommand{\EEE}{{\mathcal{E}}}

\newcommand{\BBB}{{\mathcal{B}}}

\newcommand{\sign}{{\mathrm{sign}}}
\newcommand{\lis}{{\mathrm{lis}}}

\newcommand{\Ind}{{\mathrm{Ind}}}

\newcommand{\FFF}{{\mathcal{F}}}

\newcommand{\RRRR}{{\mathcal{R}}}

\newcommand{\III}{{\mathcal{I}}}

\newcommand{\SSS}{{\mathcal{S}}}

\newcommand{\CC}{{\mathbb{C}}}
\newcommand{\QQ}{{\mathbb{Q}}}

\newcommand{\II}{{\mathbf{I}}}
\newcommand{\ZZ}{{\mathbb{Z}}}
\newcommand{\F}{{\mathbb{F}}}

\newcommand{\gr}{{\mathrm{gr}}}

\newcommand{\xx}{{\mathbf{x}}}

\newcommand{\zz}{{\mathbf{z}}}

\newcommand{\ee}{{\mathbf{e}}}

\newcommand{\Frob}{{\mathrm{Frob}}}
\newcommand{\grFrob}{{\mathrm{grFrob}}}

\newcommand{\RRR}{{\mathbf{R}}}

\newcommand{\DDD}{{\mathfrak{D}}}

\newcommand{\Zpoints}{{\mathcal{Z}}}

\newcommand{\bump}{\mathrm{bump}}
\newcommand{\Bump}{\mathrm{Bump}}
\newcommand{\pile}{\mathrm{pile}}

\begin{document}

\title[Derangement permutation matrices and orbit harmonics]
{Derangement permutation matrices and orbit harmonics}

\author[Yupeng Li]{Yupeng Li}
\address{Michigan State University}
\email{yupengli@msu.edu}

\author[Jasper Liu]{Jasper Liu}
\address{University of California, San Diego}
\email{mol008@ucsd.edu}

\author[Brendon Rhoades]{Brendon Rhoades}
\address{University of California, San Diego}
\email{bprhoades@ucsd.edu}

\begin{abstract}
    Let $\xx_{n \times n}$ be an $n \times n$ matrix of variables and let $S = \F[\xx_{n \times n}]$ be the polynomial ring over these variables where $\F$ is a field of characteristic zero. Regard $S$ as the coordinate ring of the affine space $\F^{n \times n}$ of $n \times n$  $\F$-matrices. Let $\DDD_n \subseteq \F^{n \times n}$ be the locus of derangement permutation matrices. We study the orbit harmonics quotient ring $\RRR(\DDD_n) = S/\gr \, \II(\DDD_n)$ where $\gr \, \II(\DDD_n)$ is the associated graded ideal of the vanishing ideal $\II(\DDD_n) \subseteq S$. We give an explicit generating set of $\gr \, \II(\DDD_n),$ relate the Hilbert series of $\RRR(\DDD_n)$ to the Foata transformation and the longest increasing subsequence statistic on $\symm_n$, and give an alternating sum formula for the graded $\symm_n$-character of $\RRR(\DDD_n)$. Our proofs make heavy use of the mapping cone construction of homological algebra.
\end{abstract}

\maketitle

\section{Introduction}
\label{sec:Introduction}

Let $\symm_n$ be the symmetric group on $[n] := \{1,\dots,n\}$. A permutation $w \in \symm_n$ is a {\em derangement} if $w(i) \neq i$ for all $1 \leq i \leq n$. We write $\DDD_n \subseteq \symm_n$ for the set of derangements in $\symm_n$. In previous work, the third author \cite{RhoadesViennot} analyzed the orbit harmonics ring $\RRR(\symm_n)$ corresponding to the locus $\symm_n$ of $n \times n $ permutation matrices. In this paper we do the same for the orbit harmonics ring $\RRR(\DDD_n)$ of the locus of derangement permutation matrices.

Fix a field $\F$ of characteristic zero.
Let $\xx_{n \times n} = (x_{i,j})_{1 \leq i,j \leq n}$ be an $n \times n$ matrix of variables and let $S := \F[\xx_{n \times n}]$ be the polynomial ring over these variables with its standard grading  $\deg(x_{i,j}) = 1$ for all $1 \leq i,j \leq n$.   The product group $\symm_n \times \symm_n$ acts on the matrix $\xx_{n \times n}$ by row and column permutation, viz.
\[
(w,v) \cdot x_{i,j} := x_{w(i), v(j)} \quad \quad (w,v \in \symm_n , \, \, 1 \leq i,j \leq n).
\]
This gives a graded action of $\symm_n \times \symm_n$ on $S$.

An {\em increasing subsequence} of $w \in \symm_n$ is a sequence $1 \leq i_1 < \cdots < i_k \leq n$ of indices such that $w(i_1) < \cdots < w(i_k)$. Write $\lis(w)$ for the length of the longest increasing subsequence of $w$. 
The permutation statistic $\lis(w)$ is closely tied to the {\em Schensted correspondence} \cite{Schensted} and plays an important role in random matrix theory \cite{BDJ}. The third author introduced the following ideal $I_n \subseteq S$ whose algebraic properties are tied to the combinatorics of increasing subsequences.

\begin{definition}
    \label{def:I-ideal-definition}
    {\em (Rhoades \cite{RhoadesViennot})}
    Let $I_n \subseteq S$ be the ideal generated by $\dots$
    \begin{itemize}
        \item all sums $x_{i,1} + \cdots + x_{i,n}$ of variables in row $i$ for $1 \leq i \leq n$,
        \item all sums $x_{1,j} + \cdots + x_{n,j}$ of variables in column $j$ for $1 \leq j \leq n$,
        \item all products $x_{i,j} \cdot x_{i,j'}$ of variables in the same row for $1 \leq i,j,j' \leq n$, and
        \item all products $x_{i,j} \cdot x_{i',j}$ of variables in the same column for $1 \leq i,i',j \leq n$.
    \end{itemize}
\end{definition}

The ideal $I_n \subseteq S$ is  homogeneous so that $S/I_n$ is a graded $\F$-algebra. Furthermore, the generating set of $I_n$ is stable under the action of $\symm_n \times \symm_n$ on $S$ so that $S/I_n$ is a graded $(\symm_n \times \symm_n)$-module.

The link between the quotient ring $S/I_n$ and permutation combinatorics comes through a ubiquitous technique in deformation theory called {\em orbit harmonics}.  Let $\xx_N = \{x_1,\dots,x_N\}$ be a finite variable set and regard the polynomial ring $\F[\xx_N] := \F[x_1,\dots,x_N]$ as the coordinate ring of polynomial functions $\F^N \to \F$. Given a finite locus $\Zpoints \subseteq \F^N$, we write $\II(\Zpoints) \subseteq \F[\xx_N]$ for its vanishing ideal
\[
\II(\Zpoints) := \{ f \in \F[\xx_N] \,:\, f(\zz) = 0 \text{ for all } \zz \in \Zpoints \}.
\]
We write $\gr \, \II(\Zpoints) \subseteq \F[\xx_N]$ for the associated graded ideal of $\II(\Zpoints)$. 

\begin{definition}
    \label{def:orbit-harmonics-quotient}
    The {\em (graded) orbit harmonics  ring} associated to $\Zpoints \subseteq \F^N$ is 
    \[
    \RRR(\Zpoints) := \F[\xx_N]/\gr \, \II(\Zpoints).
    \]
\end{definition}

The quotient $\RRR(\Zpoints)$ is a graded $\F$-algebra. We have an isomorphism 
\begin{equation}
    \label{eq:orbit-harmonics-isomorphism-intro}
    \RRR(\Zpoints) \cong \F[\Zpoints]
\end{equation}
of $\F$-vector spaces where $\RRR(\Zpoints)$ has the additional structure of a graded $\F$-vector space. If the locus $\Zpoints$ is stable under the action of a finite subgroup $G \subseteq GL_N(\F)$, the homogeneous ideal $\gr \, \II(\Zpoints)$ is stable under the induced action of $G$ on $\F[\xx_N]$. In this case, we may regard \eqref{eq:orbit-harmonics-isomorphism-intro} as an isomorphism of ungraded $G$-modules where $\RRR(\Zpoints)$ has the additional structure of a graded $G$-module. In geometric terms, the transition $\Zpoints \leadsto \RRR(\Zpoints)$ corresponds to the flat deformation which linearly deforms the locus $\Zpoints$ to a subscheme of affine space of degree $|\Zpoints|$ supported at the origin, as shown conceptually below.

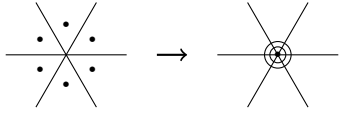
\begin{center}
 \begin{tikzpicture}[scale = 0.2]
\draw (-4,0) -- (4,0);
\draw (-2,-3.46) -- (2,3.46);
\draw (-2,3.46) -- (2,-3.46);

 \fontsize{5pt}{5pt} \selectfont
\node at (0,2) {$\bullet$};
\node at (0,-2) {$\bullet$};

\node at (-1.73,1) {$\bullet$};
\node at (-1.73,-1) {$\bullet$};
\node at (1.73,-1) {$\bullet$};
\node at (1.73,1) {$\bullet$};

\draw[thick, ->] (6,0) -- (8,0);

\draw (10,0) -- (18,0);
\draw (12,-3.46) -- (16,3.46);
\draw (12,3.46) -- (16,-3.46);

\draw (14,0) circle (15pt);
\draw(14,0) circle (25pt);
\node at (14,0) {$\bullet$};

 \end{tikzpicture}
\end{center}

The orbit harmonics deformation dates back to Kostant \cite{Kostant} in the context of coinvariant rings. It has seen application to Springer fibers \cite{GP}, delta operators \cite{HRS, Griffin}, Donaldson--Thomas theory \cite{RRT}, and Ehrhart theory \cite{RR}. For artfully chosen `combinatorial' loci $\Zpoints$, one expects algebraic properties of $\RRR(\Zpoints)$ to be governed by combinatorial properties of $\Zpoints$.

In this paper we consider the affine space $\F^{n \times n}$ of $n \times n$  matrices over $\F$. The polynomial ring $S = \F[\xx_{n \times n}]$ is the coordinate ring of $\F^{n \times n}$. We have the permutation matrix embedding $\symm_n \subseteq \F^{n \times n}$. The permutation matrix locus $\symm_n$ is stable under the action of the product group $\symm_n \times \symm_n$ by row and column permutation. The orbit harmonics ring $\RRR(\symm_n)$ may be described as follows. We write $V^\lambda$ for the irreducible representation of $\symm_n$ associated to a partition $\lambda \vdash n$.

\begin{theorem}
    \label{thm:permutation-matrix}
    {\em (Rhoades \cite{RhoadesViennot})} We have $I_n = \gr \, \II(\symm_n)$ as ideals in $S$ so that 
    \[
    S/I_n = \RRR(\symm_n).
    \]
    The quotient $\RRR(\symm_n)$ has Hilbert series
    \begin{equation}
        \Hilb(\RRR(\symm_n);q) = \sum_{w \in \symm_n} q^{n - \lis(w)}.
    \end{equation}We have
    $\RRR(\symm_n) \cong \F[\symm_n]$
    as ungraded $(\symm_n \times \symm_n)$-modules and the degree $d$ piece has isomorphism type
    \[
        \RRR(\symm_n)_d \cong \bigoplus_{\substack{\lambda \vdash n \\ \lambda_1 = n-d}} V^\lambda \otimes V^\lambda.
    \]
\end{theorem}

Motivated by Theorem~\ref{thm:permutation-matrix}, various authors \cite{Liu, LMRZ, LZ, OR, Zhu} studied the orbit harmonics rings $\RRR(\Zpoints)$ of other matrix loci $\Zpoints$. In this paper we study the locus $\DDD_n \subseteq \symm_n$ of derangement permutation matrices.
The permutation matrix embedding $\DDD_n \subseteq \symm_n \subseteq \F^{n \times n}$ gives a homogeneous ideal $\gr \, \II(\DDD_n) \subseteq S$ and an orbit harmonics quotient ring 
\[
\RRR(\DDD_n) = S/\gr \, \II(\DDD_n).
\]
Although the locus $\DDD_n$ is not closed under the row and column permuting action of the full product group $\symm_n \times \symm_n$, it is closed under the action of the diagonal subgroup $\symm_n \subseteq \symm_n \times \symm_n$ via $w \cdot v := w v w^{-1}$ for $w \in \symm_n$ and $v \in \DDD_n$.
The ideal $\gr \, \II(\DDD_n) \subseteq S$ is therefore stable under the $\symm_n$-action on $S$ given by
\[
w \cdot x_{i,j} := x_{w(i), w(j)} \quad \quad (w \in \symm_n, \, \, 1 \leq i,j \leq n)
\]
and the quotient $\RRR(\DDD_n)$ is a graded $\symm_n$-module. 

The study of $\RRR(\DDD_n)$ will be inductively  facilitated by a larger family of quotient rings. Let $I$ and $J$ be finite sets and consider the {\em board} $\BBB = I \times J$. A subset $\RRRR \subseteq \BBB$ is a {\em (nonattacking) rook placement} if 
\[
i = i' \quad \Leftrightarrow \quad j = j' \quad \quad \quad \text{for all } (i,j), (i',j') \in \RRRR.
\]
Elements of $\RRRR$ are called {\em rooks}.
An example rook placement $\RRRR = \{(2,2), (3,1), (5,4)\}$ is shown in Cartesian coordinates on the board $\BBB = [5] \times [5]$ on the left of Figure~\ref{fig:rook}.

Let $\BBB = I \times J$ be a board, let $\RRRR \subseteq \BBB$ be a rook placement, and let $(i_0,j_0) \in \RRRR$ be a rook. The {\em deletion} $\RRRR \setminus (i_0,j_0)$ is the rook placement on the same board $\BBB$ obtained by removing the rook $(i_0,j_0)$. The {\em contraction} $\RRRR/(i_0,j_0)$ is the rook placement on the smaller board $\overline{\BBB} = (I \setminus i_0) \times (J \setminus j_0)$ obtained by removing the rook $(i_0,j_0)$ together with its row and column. If $\RRRR \subseteq [5] \times [5]$ is the rook placement on the left of Figure~\ref{fig:rook}, the middle and right of Figure~\ref{fig:rook} show the deletion $\RRRR \setminus (i_0,j_0)$ and contraction $\RRRR / (i_0,j_0)$ for $(i_0,j_0) = (3,1)$.

Given a rook placement $\RRRR \subseteq [n] \times [n]$, we write
\begin{equation}
    \symm_n(\RRRR) := \{ w \in \symm_n \,:\, (i,w(i)) \notin \RRRR \text{ for all } 1 \leq i \leq n \}
\end{equation}
for the set of permutations in $\symm_n$ whose graphs avoid $\RRRR$. For example, we have
\begin{equation}
    \symm_n(\RRRR) = \DDD_n \quad \quad \text{for } \RRRR = \{(1,1), (2,2), \dots, (n,n)\}.
\end{equation}
The orbit harmonics quotient rings $\RRR(\symm_n(\RRRR))$ therefore give a generalization of $\RRR(\DDD_n)$.

The cardinalities $|\symm_n(\RRRR)|$ satisfy a deletion-contraction recursion. Suppose $\RRRR \subseteq [n] \times [n]$ is a rook placement and let $(i_0,j_0) \in \RRRR$ be a rook. We have
\begin{equation}
    \label{eq:deletion-contraction-combinatorial}
    |\symm_n(\RRRR \setminus(i_0,j_0))| = |\symm_n(\RRRR)| + |\symm_{n-1}(\RRRR/(i_0,j_0))|
\end{equation}
where we identify $\RRRR/(i_0,j_0)$ with its order-isomorphic rook placement on $[n-1] \times [n-1]$ to define $\symm_{n-1}(\RRRR/(i_0,j_0)).$ In Theorem~\ref{thm:ses-presentation} we give an algebraic lift of  Equation~\eqref{eq:deletion-contraction-combinatorial} via a short exact sequence.

\begin{figure}
\begin{tikzpicture}[scale=0.5]
  \draw (0,0) grid (5,5);

  \node at (0.5,-0.5) {$1$};
  \node at (1.5,-0.5) {$2$};
  \node at (2.5,-0.5) {$3$};
  \node at (3.5,-0.5) {$4$};
  \node at (4.5,-0.5) {$5$};

  \node at (-0.5,0.5) {$1$};
  \node at (-0.5,1.5) {$2$};
  \node at (-0.5,2.5) {$3$};
  \node at (-0.5,3.5) {$4$};
  \node at (-0.5,4.5) {$5$};

  \node at (1.5,1.5) {\large $\rook$};
  \node at (2.5,0.5) {\large $\rook$};
  \node at (4.5,3.5) {\large $\rook$};

 \draw (10,0) grid (15,5);

 \node at (10.5,-0.5) {$1$};
  \node at (11.5,-0.5) {$2$};
  \node at (12.5,-0.5) {$3$};
  \node at (13.5,-0.5) {$4$};
  \node at (14.5,-0.5) {$5$};

  \node at (9.5,0.5) {$1$};
  \node at (9.5,1.5) {$2$};
  \node at (9.5,2.5) {$3$};
  \node at (9.5,3.5) {$4$};
  \node at (9.5,4.5) {$5$};

  \node at (11.5,1.5) {\large $\rook$};
  \node at (14.5,3.5) {\large $\rook$};

   \draw (20,0) grid (24,4);

   \node at (20.5,-0.5) {$1$};
  \node at (21.5,-0.5) {$2$};
  \node at (22.5,-0.5) {$4$};
  \node at (23.5,-0.5) {$5$};
  
   \node at (19.5,0.5) {$2$};
  \node at (19.5,1.5) {$3$};
  \node at (19.5,2.5) {$4$};
  \node at (19.5,3.5) {$5$};

  \node at (21.5,0.5) {\large $\rook$};
  \node at (23.5,2.5) {\large $\rook$};
\end{tikzpicture}

\caption{A rook placement, a deletion, and a contraction.}
\label{fig:rook}
\end{figure}
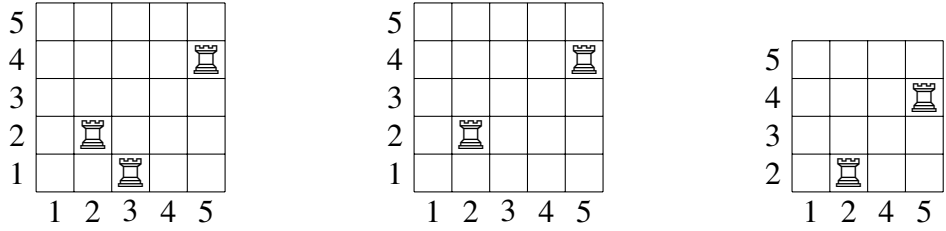

The combinatorics of $\RRR(\DDD_n)$ is tied to the {\em Foata transformation} $\Psi: \symm_n \xrightarrow{\,\, \sim \,\,} \symm_n$. Given $w \in \symm_n$, the permutation $\Psi(w)$ is obtained as follows.
\begin{enumerate}
    \item Write $w$ in {\em canonical cycle form}. That is, write each cycle of $w$ starting with its smallest element, and write the cycles of $w$ from left to right in decreasing order of smallest elements.\footnote{It is more standard to write the largest elements of cycles first, and order cycles by increasing order of largest elements. Our convention is chosen for coherence with the increasing (not decreasing) subsequence results in \cite{RhoadesViennot}.}
    \item Erase the parentheses and interpret the result as the one-line notation of $\Psi(w) \in \symm_n$.
\end{enumerate}
For example, if $w = (4,7,5) (2,3,8)(1,6) \in \symm_8$ we have $\Psi(w) = [4,7,5,2,3,8,1,6] \in \symm_8$.  This procedure is reversible and  $\Psi: \symm_n \to \symm_n$ is a bijection. The map $\Psi$ was introduced by Foata \cite{Foata} in his study of permutation statistics. Deferring various definitions to Section~\ref{sec:Background}, our main results are as follows.

\begin{enumerate}
    \item  In Theorem~\ref{thm:ses-presentation} we give an explicit generating set of the defining ideal $\gr \, \II(\symm_n(\RRRR))$ of $\RRR(\symm_n(\RRRR))$ for any rook placement $\RRRR \subseteq [n] \times [n]$. We  have \begin{equation}\gr \, \II(\symm_n(\RRRR)) = I_n + (x_{i,j} \,:\, (i,j) \in \RRRR) \quad \quad \text{ for all } n \geq 2.
    \end{equation}
    \item  We show in Theorem~\ref{thm:hilbert} that the Hilbert series of $\RRR(\DDD_n)$ is given by
    \begin{equation}
        \label{eq:hilbert-intro}
        \Hilb(\RRR(\DDD_n);q) = \sum_{w \in \DDD_n} q^{n - \lis(\Psi(w))}.
    \end{equation}
    \item We calculate the graded $\symm_n$-module structure of $\RRR(\DDD_n)$ in Theorem~\ref{thm:graded-module}. In particular, if $*$ is the Kronecker product on symmetric functions we have
    \begin{equation}
    \label{eq:frobenius-intro}
    \grFrob(\RRR(\DDD_n);q) = (-1)^n \cdot q^{n-1} \cdot e_n + \sum_{k=0}^{n-1} (-q)^k \cdot e_k \cdot \left[ \sum_{\lambda \vdash n-k} q^{n-k-\lambda_1} \cdot (s_\lambda * s_\lambda) \right]
    \end{equation}
    where $s_\lambda$ is a Schur function and $e_k = s_{(1^k)}$ is an elementary symmetric function. The authors do not know a combinatorial proof that this expression is Schur-positive.
\end{enumerate}

A standard inclusion-exclusion argument shows that derangements are counted by
\begin{equation}
    \label{eq:intro-inclusion-exclusion}
    |\DDD_n| = \sum_{k=0}^n (-1)^{k} \cdot \frac{n!}{k!}.
\end{equation}
The alternating expression \eqref{eq:frobenius-intro} is an $\symm_n$-equivariant $q$-analog of \eqref{eq:intro-inclusion-exclusion}. Taking the Hall inner product of \eqref{eq:frobenius-intro} with $e_1^n$ extracts graded dimension, and applying \eqref{eq:hilbert-intro} gives the polynomial identity
\begin{equation}
    \label{eq:alternating-q-identity}
    \sum_{w \in \DDD_n} q^{n - \lis(\Psi(w))} = 
    (-1)^n \cdot q^{n-1} + \sum_{k=0}^{n-1} (-1)^k \cdot  \left[ \sum_{\lambda \vdash n-k} q^{n-\lambda_1} \cdot \binom{n}{k} \cdot(f^\lambda)^2 \right] 
\end{equation}
where $f^\lambda$ is the number of standard Young tableaux of shape $\lambda$. Equation~\eqref{eq:alternating-q-identity} is a $q$-analog of Equation~\eqref{eq:intro-inclusion-exclusion}.

In many orbit harmonics papers, the algebraic structure of $\RRR(\Zpoints)$ is analyzed using Gr\"obner theory. By contrast, our proofs use ideas from homological algebra. The Hilbert series $\Hilb(\RRR(\DDD_n);q)$ and our presentation of $\RRR(\DDD_n)$ are established via a short exact sequence in Theorem~\ref{thm:ses-presentation}. The formula \eqref{eq:frobenius-intro} for graded $\symm_n$-character of $\RRR(\DDD_n)$ arises from an exact sequence
\[
0 \to V_n \to V_{n-1} \to \cdots \to V_0 \to \RRR(\DDD_n) \to 0
\]
of graded $\symm_n$-modules where $V_k$ is built out of copies of $\RRR(\symm_{n-k})$. This  exact sequence comes from the {\em mapping cone} construction of homological algebra.

Our paper is not the first to study derangements algebraically. D\'esarm\'enien--Wachs introduced \cite{DW} and Reiner--Webb studied \cite{RW} a character $\chi_n$ of $\symm_n$ of dimension $|\DDD_n|$ which is related to the bar resolution in homological algebra. The ungraded $\symm_n$-module $\F[\DDD_n]$ coming from the conjugation action of $\symm_n$ on $\DDD_n$ does {\bf not} have character $\chi_n$; see the remarks following Proposition~\ref{prop:ungraded-structure}. It may be interesting to find a `naturally occurring' graded refinement of $\chi_n$, but (see after Proposition~\ref{prop:ungraded-structure}) such a refinement cannot come from an $\symm_n$-stable polynomial ring quotient.

The rest of the paper is organized as follows. In {\bf Section~\ref{sec:Background}} we give background on combinatorics, representation theory, orbit harmonics, and homological algebra. The technical {\bf Section~\ref{sec:EAF}} uses filtrations to establish a certain monomorphism of quotient rings. In {\bf Section~\ref{sec:Hilbert}} we give our presentation of $\RRR(\DDD_n)$ and describe its Hilbert series. {\bf Section~\ref{sec:Module}} calculates the graded $\symm_n$-module structure of $\RRR(\DDD_n)$. We conclude in {\bf Section~\ref{sec:Conclusion}} with a conjectural basis of $\RRR(\DDD_n)$ related to the patience sorting algorithm.

\section{Background}
\label{sec:Background}

\subsection{Combinatorics}
Let $n$ be a nonnegative integer. A {\em partition} of $n$ is a weakly decreasing sequence $\lambda = (\lambda_1 \geq \cdots \geq \lambda_k)$ of positive integers with $\lambda_1 + \cdots+ \lambda_k = n$. We write $\lambda \vdash n$ to indicate that $\lambda$ is a partition of $n$.

A subset $\RRRR \subseteq [n] \times [n]$ is a {\em (nonattacking) rook placement} if for any $(i,j), (i',j') \in \RRRR$ we have
$i = i'$ if and only if $j = j'$.
Elements of $\RRRR$ are sometimes called {\em rooks} and the set $[n] \times [n]$ is sometimes called a {\em board}.

Let $\RRRR \subseteq [n] \times [n]$ be a rook placement and let $(i_0,j_0) \in \RRRR$ be a rook. The {\em deleted} rook placement $\RRRR \setminus (i_0,j_0)$ on the board $[n] \times [n]$ is given by removing the rook $(i_0,j_0)$. The {\em contracted} rook placement is the subset of the smaller board
\[
\RRRR / (i_0,j_0) \subseteq \{1,\dots,\widehat{i_0}, \dots, n \} \times \{1,\dots,\widehat{j_0}, \dots, n \}
\]
(where the hats denote omission) obtained by removing the rook $(i_0,j_0)$ and the row and column containing that rook.

We write $\Lambda = \bigoplus_{n \geq 0} \Lambda_n$ for the graded ring of symmetric functions in an infinite variable set $\xx = (x_1,x_2,\dots)$ over the ground field $\QQ(q)$. The algebra $\Lambda$ is freely generated by the {\em elementary symmetric functions} $\{ e_n \,:\, n \geq 1\}$ where
\begin{equation}
    e_n := \sum_{1 \leq i_1  < \cdots < i_n} x_{i_1} \cdots x_{i_n}
\end{equation}
is the sum of squarefree monomials of degree $n$. Vector space bases of $\Lambda_n$ are indexed by partitions $\lambda \vdash n$. We mostly use the {\em Schur basis} $\{ s_\lambda \,:\, \lambda \vdash n \}$; see e.g. \cite{Sagan} for its definition.

\subsection{\texorpdfstring{$\symm_n$}--representation theory}
Let $n \geq 0$. Irreducible representations of $\symm_n$ over the characteristic zero field $\F$ are indexed by partitions of $n$. If $\lambda \vdash n$ is a partition, we write $V^\lambda$ for the associated $\symm_n$-irreducible. For example, $V^{(n)}$ is the trivial representation $\mathrm{triv}_{\symm_n}$ and $V^{(1^n)}$ is the sign representation $\sign_{\symm_n}$.

Let $V$ be a finite-dimensional $\symm_n$-module. There exist unique multiplicities $c_\lambda \geq 0$ so that $V \cong \bigoplus_{\lambda \vdash n} c_\lambda \cdot V^\lambda$. The {\em Frobenius characteristic} of $V$ is the symmetric function
\begin{equation}
    \Frob(V) := \sum_{\lambda \vdash n} c_\lambda \cdot s_\lambda \in \Lambda_n
\end{equation}
obtained by replacing each irreducible $V^\lambda$ appearing in $V$ with the corresponding Schur function $s_\lambda$. For example, we have $\Frob(\mathrm{triv}_{\symm_n}) = s_{(n)}$ and $\Frob(\sign_{\symm_n}) = s_{(1^n)} = e_n$. More generally, if $V = \bigoplus_{i \geq 0} V_i$ is a graded $\symm_n$-module, the {\em graded Frobenius characteristic} is 
\begin{equation}
    \grFrob(V;q) := \sum_{i \geq 0} \Frob(V_i) \cdot q^i.
\end{equation}

Let $V_1$ be an $\symm_{n_1}$-module and $V_2$ be a $\symm_{n_2}$-module with $n = n_1 + n_2$. The product group $\symm_{n_1} \times \symm_{n_2}$ acts on the tensor product $V_1 \otimes V_2$ by the rule
\[
(w_1, w_2) \cdot (v_1 \otimes v_2) := (w_1 \cdot v_1) \otimes (w_2 \cdot v_2) 
\]
for all $w_1 \in \symm_{n_1}, w_2 \in \symm_{n_2}, v_1 \in V_1,$ and $v_2 \in V_2$. The {\em induction product} $V_1 \circ V_2$ of $V_1$ and $V_2$ is the $\symm_n$-module given by
\begin{equation}
    V_1 \circ V_2 := \Ind_{\symm_{n_1} \times \symm_{n_2}}^{\symm_n} (V_1 \otimes V_2).
\end{equation}
The Frobenius image of $V_1 \circ V_2$ is related to those of $V_1$ and $V_2$ by
\begin{equation}
    \Frob(V_1 \circ V_2) = \Frob(V_1) \cdot \Frob(V_2).
\end{equation}

Now let $V_1$ and $V_2$ be two $\symm_n$-modules. The tensor product $V_1 \otimes V_2$ carries an $\symm_n$-module structure given by
\[
w \cdot (v_1 \otimes v_2) := (w \cdot v_1) \otimes (w \cdot v_2) \quad \quad (w \in \symm_n, \, v_1 \in V_1, \, v_2 \in V_2).
\]
The {\em Kronecker product} on $\Lambda_n$ is the bilinear operation $*: \Lambda_n \times \Lambda_n \to \Lambda_n$ characterized by 
\begin{equation}
    \Frob(V_1 \otimes V_2) = \Frob(V_1) * \Frob(V_2)
\end{equation}
for all $\symm_n$-modules $V_1$ and $V_2$.

\subsection{Orbit harmonics}
 In this section we let $\xx_N := \{x_1,\dots,x_N\}$ be a finite set of variables and let $S := \F[\xx_N]$ be the polynomial ring over these variables equipped with its standard grading induced by $\deg(x_i) = 1$ for $1 \leq i \leq N$. The $\F$-algebra $S$ carries a graded action of $GL_N(\F)$ by linear substitutions.

Let $I \subseteq S$ be a homogeneous ideal. The quotient ring $S/I = \bigoplus_{i \geq 0} (S/I)_i$ is a graded $\F$-vector space with each graded piece $(S/I)_i$ finite-dimensional. The {\em Hilbert series} of $S/I$ is the formal power series
\begin{equation}
    \Hilb(S/I;q) := \sum_{i \geq 0} \dim_\F(S/I)_i \cdot q^i
\end{equation}
where $q$ is a formal variable.

Let $f \in S$ be a nonzero polynomial. We write $\tau(f)$ for the top-degree homogeneous component of $f$. That is, if $f= f_d + \cdots + f_1 + f_0$ where $f_i$ is homogeneous of degree $i$ and $f_d \neq 0$ we have $\tau(f) = f_d$. If $I \subseteq S$ is an ideal, the {\em associated graded ideal} is 
\begin{equation}
    \gr \, I := (\tau(f) \,:\, f \in I , \,\, f \neq 0).
\end{equation}
Observe that $\gr \, I \subseteq S$ is automatically a homogeneous ideal, even when $I$ is not homogeneous.

We identify $S = \F[\xx_N]$ with the coordinate ring of the affine space $\F^N$. Given a finite locus $\Zpoints \subseteq \F^N$, as in the introduction we write 
\begin{equation}
    \II(\Zpoints) := \{ f \in S \,:\, f(\zz) = 0 \text{ for all } \zz \in \Zpoints \}
\end{equation}
for the vanishing ideal of $\Zpoints$ and let
\begin{equation}
    \RRR(\Zpoints) := S / \gr \, \II(\Zpoints)
\end{equation}
be the (graded) {\em orbit harmonics ring} associated to $\Zpoints$.  If $\Zpoints$ is stable under the action of a finite matrix group $G \subseteq GL_N(\F)$, the ideal $\gr \, \II(\Zpoints) \subseteq S$ is stable under the induced graded action of $G$ on $S$, so that $\RRR(\Zpoints)$ is a  graded $G$-module.

\begin{theorem}
    \label{thm:orbit-harmonics} 
    {\em (Fundamental Theorem of Orbit Harmonics)}
    We have an isomorphism of $\F$-vector spaces
    \begin{equation}
    \label{eq:orbit-harmonics-isomorphism}
        \F[\Zpoints] \cong \RRR(\Zpoints)
    \end{equation}
    where $\RRR(\Zpoints)$ has the additional structure of a graded $\F$-vector space. If $\Zpoints$ is stable under the action of a finite matrix group $G \subseteq GL_N(\F)$, we may regard \eqref{eq:orbit-harmonics-isomorphism} as an isomorphism of $G$-modules where $\RRR(\Zpoints)$ has the additional structure of a graded $G$-module. 
\end{theorem}

When $\Zpoints = \varnothing$ is the empty locus, we have $\RRR(\Zpoints) = \RRR(\varnothing)=0$. If $\Zpoints = \{*\}$ consists of a single point, then $\RRR(\Zpoints) = \RRR(*) = \F$ is the ground field. We allow for our affine space $\F^N$ to satisfy $N = 0$, in which case $S = \F$ and $\RRR(\Zpoints) = 0$ or $\F$ according to whether $\Zpoints \subseteq \F^0 = \{*\}$ is empty or the singleton $\{*\}$. In particular, we have
\begin{equation}
\label{eq:off-by-one}
    \RRR(\symm_0) = \RRR(\symm_1) = \F, \quad \RRR(\DDD_0) = \F, \quad \text{and} \quad \RRR(\DDD_1) = 0
\end{equation}
where $\F = S/(S_{> 0})$. This small $n$ behavior leads to various boundary conditions in our arguments. We record a simple property of orbit harmonics quotients for later use.

\begin{lemma}
    \label{lem:augment-slice}
    Let $\Zpoints \subseteq \F^N$ be a finite locus, let $a \in \F$, and consider the locus $\Zpoints \times a \subseteq \F^N \times \F$. Write $\F[\xx_N] \otimes \F[y]$ for the coordinate ring of $\F^N \times \F$. As ideals in $\F[\xx_N] \otimes \F[y]$ we have
    \[
    \gr \, \II(\Zpoints \times a ) = (\gr \, \II(\Zpoints) \otimes \F[y]) + (\F[\xx_N] \otimes (y))
    \]
    and consequently we have an isomorphism
    \[
    \RRR(\Zpoints) \xrightarrow{\, \sim \, } \RRR(\Zpoints \times a)
    \]
    induced by the map $\F[\xx_N] \to \F[\xx_N] \otimes \F[y]$ given by $f \mapsto f \otimes 1$.
\end{lemma}

\begin{proof}
    For any $f \in \II(\Zpoints)$ we have $f = f \otimes 1 \in \II(\Zpoints \times a)$ so that $\tau(f) \in \gr \, \II(\Zpoints \times a)$ and $\gr \, \II(\Zpoints) \otimes \F[y] \subseteq \gr \, \II(\Zpoints \times a)$. We have $y - a \in \II(\Zpoints \times a)$ so that $\tau(y-a) = y \in \gr \, \II(\Zpoints \times a)$ and $\F[\xx_N] \otimes (y) \subseteq \gr \, \II(\Zpoints \times a )$. We conclude that
    \begin{equation}
    \label{eq:slice-containment}
    \gr \, \II(\Zpoints \times a ) \supseteq (\gr \, \II(\Zpoints) \otimes \F[y]) + (\F[\xx_N] \otimes (y)) \quad \text{as ideals in $\F[\xx_N] \otimes \F[y].$}
    \end{equation}
    On the other hand, let $\BBB \subseteq \F[\xx_N]$ be an arbitrary set of polynomials which descends to a basis of $\RRR(\Zpoints) = \F[\xx_N]/\gr \, \II(\Zpoints)$; we have $|\BBB| = \dim_\F\RRR(\Zpoints) = |\Zpoints|$. Then $\BBB$ descends to a spanning set of 
    \[
    \frac{\F[\xx_N] \otimes \F[y]}{(\gr \, \II(\Zpoints) \otimes \F[y]) + (\F[\xx_N] \otimes (y))}
    \]
    so that
    \begin{equation*}
    |\Zpoints| = |\Zpoints \times a| = \dim_\F \RRR(\Zpoints \times a) \leq \dim_\F \frac{\F[\xx_N] \otimes \F[y]}{(\gr \, \II(\Zpoints) \otimes \F[y]) + (\F[\xx_N] \otimes (y))} \\
    \leq |\BBB| = |\Zpoints|
    \end{equation*}
    where the first inequality follows from \eqref{eq:slice-containment} since $\RRR(\Zpoints \times a) = (\F[\xx_N] \otimes \F[y])/\gr \, \II(\Zpoints \times a)$. We conclude that the containment \eqref{eq:slice-containment} of ideals is an equality and proves the first part of the lemma. The second part of the lemma follows from the first.
\end{proof}

\subsection{Filtrations} Let $A$ be an $\F$-algebra. A {\em filtration} $F_\bullet A$ of $A$ is an increasing chain of $\F$-linear subspaces
\begin{equation}
  F_0 A \subseteq F_1 A \subseteq F_2 A \subseteq \cdots 
\end{equation}
such that
\begin{equation}
    \bigcup_{d \geq 0} F_d A = A \quad \text{and} \quad F_i A \cdot F_j A \subseteq F_{i+j} A \text{ for all $i,j$}.
\end{equation}
We set $F_d A = 0$ whenever $d < 0$. An algebra equipped with a filtration is called a {\em filtered algebra}.

Let $A$ be a filtered $\F$-algebra with filtration $F_\bullet A$. The {\em associated graded algebra} is the direct sum of $\F$-vector spaces
\begin{equation}
    \gr \, A = \bigoplus_{d \geq 0} (\gr \, A)_d \quad \text{where} \quad (\gr\, A)_d := F_dA / F_{d-1}A.
\end{equation}
Multiplication is defined by the rule
\[
(f + F_{i-1} A) \cdot (g + F_{j-1} A) := (f \cdot g) + F_{i+j-1} A
\]
for all $f \in F_i A$ and $g \in F_j A$.  

Let $A$ and $B$ be filtered algebras with filtrations $F_\bullet A$ and $F_\bullet B$. Let $d \geq 0$, let $\varphi: A \to B$ be an $\F$-linear map, and assume
\begin{equation}
    \varphi( F_i A) \subseteq F_{i+d} B \quad \quad \text{for all $i$}.
\end{equation}
Then $\varphi$ induces an $\F$-linear map 
$\gr \, \varphi: \gr \, A \to \gr \, B$ defined by
\begin{equation}
    \gr \, \varphi: f + F_{i-1} A \mapsto \varphi(f) + F_{i+d-1} B \quad \quad \text{for all $f \in F_i A.$}
\end{equation}
The map $\gr \, \varphi$ is homogeneous of degree $d$. Certain nice properties of $\varphi$ can be lost upon passage to its associated graded map $\gr \, \varphi$. For example, even if $\varphi$ is injective, it is {\bf not} necessarily true that $\varphi$ is injective.  This phenomenon is the source of most of the technical difficulties in Section~\ref{sec:EAF}.

Let $S = \F[\xx_N]$ be the polynomial ring over $\xx_N = \{x_1,\dots,x_N\}$. The filtered algebras considered in this paper will have the form $R = S/I$ where $I \subseteq S$ is a (usually inhomogeneous) ideal. The subspace $F_d R \subseteq R$ is given by
\begin{equation}
    F_d R = \text{image of } S_{\leq d} \text{ in } R 
\end{equation}
where $S_{\leq d} \subseteq S$ is the subspace of polynomials of degree $\leq d$. In particular, this gives a filtered algebra structure on the ungraded orbit harmonics rings $S/\II(\Zpoints)$ for finite loci $\Zpoints \subseteq \F^N$. We include the following standard result on the associated graded rings of these algebras.

\begin{lemma}
    \label{lem:associated-graded-isomorphism}
    Let $S = \F[\xx_N]$ and let $I \subseteq S$ be an ideal and let $R = S/I$ be filtered as above. There is a canonical ring isomorphism
    \begin{equation}
        \varphi: S/ \gr \, I \xrightarrow{ \, \sim \, } \gr \, R.
    \end{equation}
\end{lemma}

\begin{proof}
    For all $d \geq 0$, let $S_d \subseteq S$ be the vector space of homogeneous polynomials of degree $d$. We have a natural $\F$-linear map
    \begin{equation}
        \varphi_d : S_d \to F_d R / F_{d-1} R
    \end{equation}
    given by $\varphi_d(f) := (f + I)+ F_{d-1} R$. The map $\varphi_d$ is surjective since for any $g \in S_{\leq d}$ we have $(g + I) + F_{d-1} R = 0$ unless $\deg(g) = d$, in which case $\varphi_d: \tau(g) \mapsto g + F_{d-1} R$.  The maps $\varphi_d$ assemble to give a linear surjection
    \begin{equation}
        \varphi: S \to \gr_F \, R
    \end{equation}
    with $\varphi = \bigoplus_{d \geq 0} \varphi_d$.
    One checks that $\varphi$ commutes with multiplication and is an $\F$-algebra epimorphism. 
    
    It suffices to show that $\ker(\varphi) = \gr \, I$. We have
    \begin{equation}
        \ker(\varphi) = \bigoplus_{d \geq 0} \ker(\varphi_d).
    \end{equation}
    Let $d \geq 0$, let $f \in S_d$ be a nonzero homogeneous polynomial, and assume $f \in \ker(\varphi_d)$. Then $f + I \in F_{d-1} R$, so there exists some $h \in S_{\leq d-1}$ so that $f + h \in I$. But $f = \tau(f + h) \in \gr \, I$ so that $f \in \ker(\varphi_d)$. We conclude that $\ker(\varphi) \subseteq \gr \, I$. 
    For the other containment, assume $f \in S_d$ is homogeneous and nonzero and that $f \in \gr \, I$. There exists $g \in I$ so that $f = \tau(g)$. We have $(f-g) \in S_{\leq d-1}$ so that $(f-g) + I \in F_{d-1} R$. We have
    \[
    \varphi(f) = (f + I) + F_{d-1} R = (g + I) + F_{d-1} R = 0
    \]
    so that $f \in \ker(\varphi).$
\end{proof}

\subsection{Homological algebra} The mapping cone construction of homological algebra will play a crucial role in our analysis of $\RRR(\DDD_n)$. We review this construction here and refer the reader to \cite{Weibel} for a textbook treatment.
For simplicity, we work in the category of $\F$-vector spaces, but the constructions and results given here hold in any abelian category.

A ($\ZZ_{\geq 0}$-indexed) {\em chain complex} $(A_\bullet,d)$ is a sequence $(A_n)_{n \geq 0}$ of $\F$-vector spaces equipped with linear maps 
\begin{equation}
\cdots \to A_{n+1} \xrightarrow{\, d_{n+1} \, } A_n \xrightarrow{ \, d_n \, } A_{n-1} \xrightarrow{\, d_{n-1} \, } \cdots \to A_1 \xrightarrow { \, d_1 \, } A_0 \to 0
\end{equation}
such that
\begin{equation}
    d_{n-1} \circ d_n = 0 \quad \quad \text{for all } n \geq 1.
\end{equation}
The vector spaces $A_n$ are called {\em chain groups}. 
Given two chain complexes $(A_\bullet, d^A)$ and $(B_\bullet,d^B)$, a {\em chain map} $f_\bullet: A_\bullet \to B_\bullet$ is a family of linear maps $f_n: A_n \to B_n$ such that the diagram
\[
    \begin{tikzpicture}[scale = 0.8]
            \node(An) at (0,0) {$A_n$};
            \node(Anm) at (4,0) {$A_{n-1}$};
            \node(Bn) at (0,-2) {$B_n$};
            \node(Bnm) at (4,-2) {$B_{n-1}$}; 

    \draw[->] (An) -- node[above] {$d_n^A$} (Anm);
    \draw[->] (Bn) -- node[below] {$d_n^B$} (Bnm);
    \draw[->] (An) -- node[left] {$f_n$} (Bn);
    \draw[->] (Anm) -- node[right] {$f_{n-1}$} (Bnm);
    \end{tikzpicture}
\]
commutes for all $n \geq 1$.

Let $(A_\bullet, d^A)$ and $(B_\bullet, d^B)$ be chain complexes and let $f_\bullet: A_\bullet \to B_\bullet$ be a chain map. The {\em mapping cone} $\Cone(f_\bullet)$ is a new complex whose chain groups are given by
\begin{equation}
    \Cone(f_\bullet)_n :=  B_n \oplus A_{n-1}
\end{equation}
for $n \geq 0$. (Here we interpret $A_{-1} := 0$.) The differential 
$d_n: \Cone(f_\bullet)_n \to \Cone(f_\bullet)_{n-1}$ is given by the formula
\begin{equation}
    d_n( b, a) := (d_{n}^B(b) + f_{n-1}(a), - d_{n-1}^A(a))
\end{equation}
for all $b \in B_n$ and $a \in A_{n-1}$. It can be shown that $d_{n-1} \circ d_n = 0$ for all $n \geq 1$ so that $\Cone(f_\bullet)$ is in fact a chain complex.

We will use mapping cones to inductively construct resolutions of objects in short exact sequences. Let 
\begin{equation}
    0 \to A  \xrightarrow{ \, \, \iota \, \, } B  \xrightarrow{ \, \, \pi \, \, } C \to 0
\end{equation}
be a short exact sequence of $\F$-vector spaces. Suppose we have exact sequences 
\begin{equation}
\label{eq:pq-resolution}
    \cdots \to P_2 \xrightarrow{\, d_2^P \, } P_1 \xrightarrow{\, d_1^P \,} P_0 \xrightarrow{\, d_0^P \,} A \to 0 \quad \text{and} \quad 
    \cdots \to Q_2 \xrightarrow{\, d_2^Q \,} Q_1 \xrightarrow{\, d_1^Q \, } Q_0 \xrightarrow{\, d_0^Q \, } B \to 0
\end{equation}
of $\F$-vector spaces. The situation of \eqref{eq:pq-resolution} is  expressed as $P_\bullet$ being a {\em resolution} of $A$ (and $Q_\bullet$ being a resolution of $B$).\footnote{For more general abelian categories, we would require that the $P_n$ and $Q_n$ be projective objects.} Suppose we have a chain map $f_\bullet: P_\bullet \to Q_\bullet$ with the additional condition that the diagram
\[
    \begin{tikzpicture}[scale = 0.8]
            \node(An) at (0,0) {$P_0$};
            \node(Anm) at (4,0) {$A$};
            \node(Bn) at (0,-2) {$Q_0$};
            \node(Bnm) at (4,-2) {$B$}; 

    \draw[->] (An) -- node[above] {$d_0^P$} (Anm);
    \draw[->] (Bn) -- node[below] {$d_0^Q$} (Bnm);
    \draw[->] (An) -- node[left] {$f_0$} (Bn);
    \draw[->] (Anm) -- node[right] {$\iota$} (Bnm);
    \end{tikzpicture}
\]
commutes. The $\Cone(f_\bullet)$ construction gives a resolution of $C$ as follows.

\begin{theorem}
    \label{thm:exact-cone}
    In the above situation, we have an exact sequence of $\F$-vector spaces given by
    \begin{equation}
      \cdots   \to \Cone(f_\bullet)_2 \xrightarrow{\, d_2 \, } \Cone(f_\bullet)_1 \xrightarrow{\, d_1 \, } \Cone(f_\bullet)_0 \to C \to 0
    \end{equation}
    where the map from $\Cone(f_\bullet)_0 = Q_0 \oplus P_{-1} = Q_0$ to $C$ is the composition
    \begin{equation}
        \Cone(f_\bullet)_0 = Q_0 \xrightarrow{\, d_0^Q \, } B \xrightarrow{\, \pi \, } C.
    \end{equation}
\end{theorem}

\section{Extension, averaging, and filtrations}
\label{sec:EAF}

Let $\RRRR \subseteq [n] \times [n]$ be a rook placement and let $(i_0,j_0) \in \RRRR$ be a rook. The goal of this section is to establish  an injection $\RRR(\symm_{n-1}(\RRRR/(i_0,j_0))) \hookrightarrow \RRR(\symm_n(\RRRR))$ induced by multiplication by $x_{i_0,j_0}$. This is done in Lemma~\ref{lem:x11-injection}. The arguments involved are technical and perhaps should be skipped on a first reading. We begin by considering arbitrary `forbidden' boards.

\subsection{Inhomogeneous ideals for arbitrary boards} Let $\BBB \subseteq [n] \times [n]$ be an arbitrary board. We write 
\begin{equation}
    \symm_n(\BBB) := \{ \text{permutation matrices of } w \in \symm_n \,:\, w(j) \neq i \text{ for all } (i,j) \in \BBB \}
\end{equation}
for the set of permutation matrices in $\symm_n$ which avoid $\BBB$. The following result gives an explicit generating set of the vanishing ideal of the locus $\symm_n(\BBB)$. In the next proof and throughout, we write
\begin{equation}
    x_\RRRR := \prod_{(i,j) \in \RRRR} x_{i,j} 
\end{equation}
for any rook placement $\RRRR \subseteq [n] \times [n]$. We also say that a permutation $w \in \symm_n$ {\em extends $\RRRR$} if $w(j) = i$ for all $(i,j) \in \RRRR$.

\begin{lemma}
    \label{lem:arbitrary-board-generators}
    The vanishing ideal $\II(\symm_n(\BBB)) \subseteq S$ is generated by $\dots$
    \begin{enumerate}
        \item differences $x_{i,1} + \cdots + x_{i,n} - 1$ and $x_{1,j} + \cdots + x_{n,j} - 1$ for $1 \leq i,j \leq n$,
        \item differences $x_{i,j}^2 - x_{i,j}$ for $1 \leq i,j \leq n$,
        \item products $x_{i,j} \cdot x_{i,j'}$ and $x_{i,j} \cdot x_{i',j}$ where $i \neq i'$ and $j \neq j'$, and
        \item variables $x_{i,j}$ for $(i,j) \in \BBB$.
    \end{enumerate}
\end{lemma}

\begin{proof}
    Let $I \subseteq S$ be the ideal generated by polynomials of the form $(1)-(4)$. It is not difficult to check that each of these polynomials vanishes on $\symm_n(\BBB)$ so that
    \begin{equation}
        \label{eq:arbitrary-board-containment}
        I \subseteq \II(\symm_n(\BBB)) \quad \text{as ideals in $S$.}
    \end{equation}
    To prove that the containment \eqref{eq:arbitrary-board-containment} is an equality, it suffices to establish the dimension bound
    \begin{equation}
        \label{eq:dim-bound-board}
        \dim_\F S/I \leq |\symm_n(\BBB) | = \dim_\F \F[\symm_n(\BBB)] = \dim_\F S/\II(\symm_n(\BBB)).
    \end{equation}
    For $w \in \symm_n(\BBB)$ we write
    \begin{equation}
        e_w := \prod_{j=1}^n x_{w(j),j}.
    \end{equation}
    A direct calculation in the quotient ring $S/I$ gives 
    \begin{equation}
        \label{eq:dbb-one}
        1 = \prod_{i=1}^n 1 = \prod_{i=1}^n (x_{i,1} + \cdots + x_{i,n}) = \sum_{w \in \symm_n(\BBB)} e_w
    \end{equation}
    since the only surviving terms in the expansion of $\prod_{i=1}^n (x_{i,1} + \cdots + x_{i,n})$ in $\F[\symm_n(\BBB)]$ correspond to permutations avoiding $\BBB$.
    
    The quotient ring $S/I$ is spanned over $\F$ by the rook placement monomials
    \begin{equation}
    \{ x_\RRRR \,:\, \RRRR \subseteq [n] \times [n] \text{ a rook placement} \}.
    \end{equation}
    Working in $S/I$, if $\RRRR$ is a rook placement we have
    \begin{equation}
        x_\RRRR = x_\RRRR \cdot 1 = x_\RRRR \cdot \sum_{w \in \symm_n(\BBB)} e_w = \sum_{\substack{w \in \symm_n(\BBB) \\ w \text{ extends } \RRRR}} e_w
    \end{equation}
    and we conclude that $\{ e_w \,:\, w \in \symm_n(\BBB) \}$ spans $S/I$ over $\F$. The dimension inequality \eqref{eq:dim-bound-board} follows and the proof is complete.
\end{proof}

The elements $\{ e_w \,:\, w \in \symm_n(\BBB) \}$ in $\F[\symm_n(\BBB)]$ appearing in the above proof satisfy $e_w^2 =e_w$ for all $w \in \symm_n(\BBB)$ and $e_w \cdot e_v = 0$ when $w \neq v$. Since $\sum_{w \in \symm_n(\BBB)} e_w = 1$, the $e_w$ are a complete set of orthogonal idempotents in $\F[\symm_n(\BBB)]$.

There is a natural filtration $F_\bullet$ on $\F[\symm_n(\BBB)] = S/\II(\symm_n(\BBB))$ given by
\[
F_d \F[\symm_n(\BBB)] = \text{image of $S_{\leq d}$ in } \F[\symm_n(\BBB)].
\]
The vector space $F_d \F[\symm_n(\BBB)]$ has a nice spanning set. 

\begin{lemma}
    \label{lem:board-filtration-span}
    The following collection of squarefree monomials descends to a spanning set of $F_d \F[\symm_n(\BBB)]$:
    \[
    \{ x_\RRRR \,:\, \RRRR \subseteq [n] \times [n] \text{ a rook placement and } |\RRRR| \leq d \}.
    \]
\end{lemma}

\begin{proof}
    The vector space $F_d \F[\symm_n(\BBB)]$ is spanned by cosets $m + \II(\symm_n(\BBB))$ where $m \in S$ is a monomial of degree $\leq d$. Write $m = \prod_{i,j=1}^n x_{i,j}^{a_{i,j}}$ where $a_{i,j} \geq 0$ and $\sum_{i,j} a_{i,j} \leq d$. By Lemma~\ref{lem:arbitrary-board-generators} we have
    $m + \II(\symm_n(\BBB)) = \bar{m} + \II(\symm_n(\BBB))$ where $\bar{m} \in S$ is the squarefree monomial
    \begin{equation}
    \bar{m} := \prod_{i,j = 1}^n x_{i,j}^{\bar{a}_{i,j}} \quad \quad \bar{a}_{i,j} := \begin{cases}
        1 & a_{i,j} > 0, \\
        0 & a_{i,j} = 0.
    \end{cases}
    \end{equation}
    Lemma~\ref{lem:arbitrary-board-generators} also implies that $\bar{m} \in \II(\symm_n(\BBB))$ unless $\bar{m} = x_\SSS$ for some rook placement $\SSS \subseteq [n] \times [n]$ with $|\SSS| = \sum_{i,j} \bar{a}_{i,j} \leq \sum_{i,j} a_{i,j} \leq d$.
\end{proof}

\subsection{Extension and averaging} Throughout this technical subsection, we use the following 

\begin{quote}
    {\bf Notation.} {\em Fix an integer $n \geq 2$ and a subset $T \subseteq \{2,\dots,n\}$. Let $\RRRR \subseteq [n] \times [n]$ be the rook placement 
    \[\RRRR := \{(i,i) \,:\, i \in T \}.\]}
    We have the set $\symm_n(\RRRR) \subset \F^{n \times n}$ of permutation matrices avoiding $\RRRR$. Write
    $\overline{\symm_n(\RRRR)}  \subseteq \symm_n(\RRRR)$ for the subset
    \[
    \overline{\symm_n(\RRRR)} := \{ w \in \symm_n(\RRRR) \,:\, w(1) = 1 \}.
    \]
\end{quote}

In the definition of $\overline{\symm_n(\RRRR)}$, we identify permutations with their permutation matrices. Regarding $\symm_{n-1}$ as permutations of $\{2,\dots,n\}$ gives a natural identification  $\overline{\symm_n(\RRRR)} \xrightarrow{\, \, \sim \, \, } \symm_{n-1}(\RRRR)$ induced by deleting the first row and column.

We have nested vanishing ideals $\II(\symm_n(\RRRR)) \subseteq \II(\overline{\symm_n(\RRRR)}) \subseteq S$ and ungraded quotient rings
\begin{equation}
    \F[\symm_n(\RRRR)] = S/ \II(\symm_n(\RRRR)) \quad \text{and} \quad \F[\overline{\symm_n(\RRRR)}] = S / \II(\overline{\symm_n(\RRRR)}).
\end{equation}
Both of these $\F$-algebras have a filtration induced from the standard filtration on $S$, i.e.
\begin{align*}
    F_d \F [\symm_n(\RRRR)] &= \text{image of $S_{\leq d}$ in } \F[\symm_n(\RRRR)] \text{ and } \\
    F_d \F [\overline{\symm_n(\RRRR)}] &= \text{image of $S_{\leq d}$ in } \F[\overline{\symm_n(\RRRR)}]. 
\end{align*}
We think of $\F[\symm_n(\RRRR)]$ both as the quotient ring $S/\II(\symm_n(\RRRR))$ and the $\F$-algebra of functions $f: \symm_n(\RRRR) \to \F$ with pointwise operations, and similarly for $\F[\overline{\symm_n(\RRRR)}]$. This subsection introduces `extension' and `averaging' maps
\begin{equation}
    \EEE: \F[\overline{\symm_n(\RRRR)}] \to \F[\symm_n(\RRRR)] \quad \text{and} \quad 
    \AAA: \F[\symm_n(\RRRR)] \to \F[\overline{\symm_n(\RRRR)}]
\end{equation}
which will help prove a deletion-contraction short exact sequence. The extension maps are defined using quotient rings as follows.

\begin{lemma}
    \label{lem:extension-definition}
    Multiplication by $x_{1,1}$ induces a well-defined $S$-module homomorphism
    \[
    \EEE: \F[\overline{\symm_n(\RRRR)}] \xrightarrow{\, \, x_{1,1} \cdot (-) \,  \,} \F[\symm_n(\RRRR)].
    \]
    For any integer $d$ we have
    \[
    \EEE( F_d \F[\overline{\symm_n(\RRRR)}] ) \subseteq  F_{d+1} \F[\symm_n(\RRRR)].
    \]
\end{lemma}

\begin{proof}
    To check that $\EEE$ is well-defined, we need to prove 
    \begin{equation}
    \label{eq:mult-bar-containment}
        x_{1,1} \cdot \II(\overline{\symm_n(\RRRR)}) \subseteq \II(\symm_n(\RRRR)) \quad \text{in $S$.}
    \end{equation}
    Lemma~\ref{lem:arbitrary-board-generators} gives a generating set for the ideal $\II(\symm_n(\RRRR))$. We claim that 
    \begin{equation}
        \label{eq:bar-ideal}
        \II(\overline{\symm_n(\RRRR)}) = \II(\symm_n(\RRRR)) + (x_{1,1} - 1) + (x_{i,1}, x_{1,j} \,:\, 2 \leq i,j \leq n ) \quad \text{as ideals in $S$.}
    \end{equation}
    Since $x_{1,1} - 1$ and $x_{i,1}, x_{j,1}$ vanish on $\overline{\symm_n(\RRRR)}$ for $i,j > 1$ we have
    \begin{equation}
    \label{eq:bar-ideal-containment}
        \II(\overline{\symm_n(\RRRR)}) \supseteq \II(\symm_n(\RRRR)) + (x_{1,1} - 1) + (x_{i,1}, x_{1,j} \,:\, 2 \leq i,j \leq n ).
    \end{equation}
    Using the notation of the proof of Lemma~\ref{lem:arbitrary-board-generators}, the set $\{ e_w \,:\, w \in \symm_n(\RRRR) \}$ descends to an $\F$-basis of $S/\II(\symm_n(\RRRR))$. The containment \eqref{eq:bar-ideal-containment} implies that 
    \begin{multline}
    \label{eq:long-quotient-spanning}
        \text{$\{ e_w \,:\, w \in \overline{\symm_n(\RRRR)} \}$ descends to a spanning set of the quotient ring} \\
        \frac{S}{\II(\symm_n(\RRRR)) + (x_{1,1} - 1) + (x_{i,1}, x_{1,j} \,:\, 2 \leq i,j \leq n )}.
    \end{multline}
    Combining \eqref{eq:bar-ideal-containment} and \eqref{eq:long-quotient-spanning}, we conclude the desired ideal equality \eqref{eq:bar-ideal}. The desired containment \eqref{eq:mult-bar-containment} is reduced to checking that 
    \[
        x_{1,1} \cdot (x_{1,1} - 1) \in \II(\symm_n(\RRRR))
    \]
    and that
    \[
        x_{1,1} \cdot x_{i,1}, \, x_{1,1} \cdot x_{1,j} \in \II(\symm_n(\RRRR)) \quad \text{ for all } 2 \leq i,j \leq n.
    \]
    Both of these memberships follow from Lemma~\ref{lem:arbitrary-board-generators}. This establishes \eqref{eq:mult-bar-containment}, and $\EEE$ is well-defined. It is clear that $\EEE$ carries $F_d \F[\overline{\symm_n(\RRRR)}]$ into $F_{d+1} \F[\symm_n(\RRRR)]$.
\end{proof}

The term `extension' is justified by the action of $\EEE$ on functions $f: \overline{\symm_n(\RRRR)} \to \F$. We have
\begin{equation}
    \EEE(f)(w) = \begin{cases}
        f(w) & \text{if $w(1) = 1$,} \\
        0 & \text{otherwise,}
    \end{cases}
\end{equation}
for all functions $f: \overline{\symm_n(\RRRR)} \to \F$ and all $w \in \symm_n(\RRRR)$. The map $\EEE$ is therefore extension by zero so that $\EEE: \F[\overline{\symm_n(\RRRR)}] \to \F[\symm_n(\RRRR)]$ is an injection. Lemma~\ref{lem:associated-graded-isomorphism} and Lemma~\ref{lem:extension-definition} furnish an homogeneous degree 1 map
\begin{equation}
    \gr \, \EEE: \RRR(\overline{\symm_n(\RRRR)}) \xrightarrow{\, \, x_{1,1} \cdot (-) \, \, } \RRR(\symm_n(\RRRR))
\end{equation}
between graded orbit harmonics quotient rings which is induced by multiplication by $x_{1,1}$. The injectivity of $\EEE$ does not by itself establish the injectivity of $\gr \, \EEE$. The rest of this subsection is devoted to proving that that $\gr \, \EEE$ is injective (Lemma~\ref{lem:gr-e-injective}).

The injectivity of $\gr \, \EEE$ will be proven using an `averaging operator' $\AAA: \F[\symm_n(\RRRR)] \to \F[\overline{\symm_n(\RRRR)}]$. To define $\AAA$, we need some notation. For a permutation matrix $w \in \symm_n$ and $2 \leq b \leq n$, define
\begin{equation}
    w_b := \text{permutation matrix obtained by swapping rows 1 and $b$ of $w$}.
\end{equation}
 Since $\RRRR = \{ (i,i) \,:\, i \in T \}$ and $T \subseteq \{2,\dots,n\}$ we have the following observation.

\begin{observation}
    \label{obs:b-avoid}
    Let $w \in \overline{\symm_n(\RRRR)}$ and let $2 \leq b \leq n$. We have $w_b \in \symm_n(\RRRR \cup (1,1)).$
\end{observation}

As an example of Observation~\ref{obs:b-avoid}, suppose $n = 4$, $\RRRR = \{(2,2), (3,3)\}$, and let $w \in \overline{\symm_n(\RRRR)}$ be the permutation matrix
\[
w = \begin{pmatrix}
 1 & 0 & 0 & 0 \\
 0 & \boxtimes & 1 & 0 \\
 0 & 0 & \boxtimes & 1 \\
 0 & 1 & 0 & 0 
\end{pmatrix}
\]
where the $\boxtimes$ occupy the `forbidden' positions of $\RRRR$. The condition  $w \in \overline{\symm_n(\RRRR)}$ corresponds to the 1 in the $(1,1)$-position. We have
\[
w_2 = \begin{pmatrix}
 0 & 0 & 1 & 0 \\
 1 & \boxtimes & 0 & 0 \\
 0 & 0 & \boxtimes & 1 \\
 0 & 1 & 0 & 0 
\end{pmatrix}, \quad
w_3 = \begin{pmatrix}
 0 & 0 & 0 & 1 \\
 0 & \boxtimes & 1 & 0 \\
 1 & 0 & \boxtimes & 0 \\
 0 & 1 & 0 & 0 
\end{pmatrix}, \quad \text{and} \quad 
w_4 = 
\begin{pmatrix}
 0 & 1& 0 & 0 \\
 0 & \boxtimes & 1 & 0 \\
 0 & 0 & \boxtimes & 1 \\
 1 & 0 & 0 & 0 
\end{pmatrix}.
\]
Each of $w_2, w_3, w_4$ avoid the $\boxtimes$ positions of $\RRRR$ as well as the position $(1,1)$.

The averaging map $\AAA: \F[\symm_n(\RRRR)] \to \F[\overline{\symm_n(\RRRR)}]$ is defined using functions. For a function $f: \symm_n(\RRRR) \to \F$ and $w \in \overline{\symm_n(\RRRR)}$ we set
\begin{equation}
    \AAA(f)(w) := \sum_{b=2}^n f(w_b).
\end{equation}
Observation~\ref{obs:b-avoid} implies that this gives a well-defined function $\AAA(f): \overline{\symm_n(\RRRR)} \to \F$. 

The map $\AAA: \F[\symm_n(\RRRR)] \to \F[\overline{\symm_n(\RRRR)}]$ is $\F$-linear. Its purpose is to prove that $\gr \, \EEE$ is injective. The next three lemmas establish important properties of $\AAA$.

\begin{lemma}
    \label{lem:ea-zero}
    We have $\AAA \circ \EEE= 0$ as an operator on $\F[\overline{\symm_n(\RRRR)}]$.
\end{lemma}

\begin{proof}
    For $f: \symm_n(\RRRR) \to \F$, the function $\EEE(f): \symm_n(\RRRR) \to \F$ vanishes off of $\overline{\symm_n(\RRRR)} \subseteq \symm_n(\RRRR)$. Since 
    \[
    \overline{\symm_n(\RRRR)} = \symm_n(\RRRR) \setminus \symm_n(\RRRR \cup (1,1)),
    \]
    we are done by Observation~\ref{obs:b-avoid} and the definition of $\AAA$.
\end{proof}

We want to understand how $\AAA: \F[\symm_n(\RRRR)] \to \F[\overline{\symm_n(\RRRR)}]$ interacts with filtrations. To this end, we study how $\AAA$ acts on functions $\symm_n(\RRRR) \to \F$ of the form $x_\SSS = \prod_{(i,j) \in \SSS} x_{i,j}$ where $\SSS \subseteq [n] \times [n]$ is a rook placement. The function $\AAA(x_\SSS)$ depends on how $\SSS$ interacts with the first row and column of $[n] \times [n]$.

\begin{lemma}
    \label{lem:a-on-rook-placement}
    Let $\SSS \subseteq [n] \times [n]$ be a rook placement and consider the squarefree monomial $x_\SSS$ as a function $\symm_n(\RRRR) \to \F$ so that $\AAA(x_\SSS)$ is a function $\overline{\symm_n(\RRRR)} \to \F$. As elements of $\F[\overline{\symm_n(\RRRR)}]$ we have the following identifications.
      \begin{enumerate}
        \item If $\SSS$ does not occupy the first row or first column we have
        \[\AAA(x_\SSS) = (n-1-|\SSS|) \cdot x_\SSS \quad \text{in $\F[\overline{\symm_n(\RRRR)}]$}.\] 
        \item If $(1,1) \in \SSS$ then
        \[
        \AAA(x_\SSS) = 0 \quad \text{in $\F[\overline{\symm_n(\RRRR)}]$}.
        \]
        \item If $(1,c)\in \SSS$ where $c \neq 1$ and does not involve the first column we have 
        \[\AAA(x_\SSS) =  x_{\SSS'} \text{ where } \SSS' = \SSS \setminus \{(1,c)\} \quad \text{in $\F[\overline{\symm_n(\RRRR)}]$}.\] 
        \item If $(r,1)\in \SSS$ where $r\neq 1$ and does not involve the first row we have 
        \[\AAA(x_\SSS) =  x_{\SSS'} \text{ where } \SSS' = \SSS \setminus \{(r,1)\} \quad \text{in $\F[\overline{\symm_n(\RRRR)}]$}.\]
        \item If $(1,c),(r,1)\in \SSS$ where $c,r \neq 1$ we have 
        \[
        \AAA(x_\SSS) = x_{\SSS'\cup (r,c)} \text{ where } \SSS' = \SSS \setminus \{ (1,c), (r,1) \} \quad \text{in $\F[\overline{\symm_n(\RRRR)}]$}.
        \]
    \end{enumerate}
\end{lemma}

\begin{proof}
    Each of the identities (1) -- (5) is proven by showing that both sides coincide as functions $\overline{\symm_n(\RRRR)} \to \F$. 

    (1) Let $w \in \overline{\symm_n(\RRRR)}$. Suppose first that $w$ contains $\SSS$ so that $x_\SSS(w) = 1$. Then $w_b$ contains $\SSS$ if and only if $\SSS$ does not have a rook in row $b$. We compute
    \begin{equation}
        \AAA(x_\SSS)(w) = \sum_{b=2}^n x_\SSS(w_b) = \sum_{\substack{2 \leq b \leq n \\ \SSS \text{ does not contain rooks in row $b$}}} 1  = n - 1 - |\SSS| = (n - 1 - |\SSS|) \cdot x_\SSS(w).
    \end{equation}
    On the other hand, if $w$ does not contain $\SSS$, then $w_b$ does not contain $\SSS$ for $2 \leq b \le n$ and
    \begin{equation}
        \AAA(x_\SSS)(w) = \sum_{b=2}^n x_\SSS(w_b) = \sum_{b=2}^n 0 = 0 = (n-1-|\SSS|) \cdot x_\SSS(w).
    \end{equation}

    (2) Let $w \in \overline{\symm_n(\RRRR)}$. Since $(1,1) \in \SSS$, by Observation~\ref{obs:b-avoid} the permutation matrix $w_b$ does not contain $\SSS$ for $2 \leq b \leq n$ and
    \begin{equation}
        \AAA(x_\SSS)(w) = \sum_{b=2}^n x_\SSS(w_b) = \sum_{b=2}^n 0 = 0.
    \end{equation}

    (3) Let $w \in \overline{\symm_n(\RRRR)}$. Let $2 \leq b_0 \leq n$ be the unique index with $w(c) = b_0$. Since $(1,c) \in \SSS$ we have 
    \begin{equation}
        \AAA(x_\SSS)(w) = \sum_{b=2}^n x_\SSS(w_b) = x_\SSS(w_{b_0}) = x_{\SSS'}(w_{b_0}).
    \end{equation}
    where the final equality uses $w_{b_0}(c) = 1$. Since $(1,c) \in \SSS$ and $2 \leq b_0 \leq n$, we see that $(b_0,c) \notin \SSS$ and 
    \[
    w_{b_0} \text{ contains } \SSS' \quad \Leftrightarrow \quad w \text{ contains } \SSS'.
    \]
    Thus $\AAA(x_\SSS)(w) = x_{\SSS'}(w_{b_0}) = x_{\SSS'}(w).$ 

    (4) Let $w \in \overline{\symm_n(\RRRR)}$. Since $(r,1) \in \SSS$ and $w(1) = 1$ we have
    \begin{equation}
        \AAA(x_\SSS)(w) = \sum_{b=2}^n x_\SSS(w_b) = x_\SSS(w_r) = x_{\SSS'}(w_r).
    \end{equation}
    As in the previous case, one checks that 
    \[
    w_{r} \text{ contains } \SSS' \quad \Leftrightarrow \quad w \text{ contains } \SSS'
    \]
    so that $\AAA(x_\SSS)(w) = x_{\SSS'}(w_r) = x_{\SSS'}(w)$.

    (5) Let $w \in \overline{\symm_n(\RRRR)}$. Since $(r,1) \in \SSS$ we have
    \begin{equation}
        \AAA(x_\SSS)(w) = \sum_{b=2}^n x_\SSS(w_b) = x_\SSS(w_r) = x_{\SSS \setminus \{(r,1)\}}(w_r)
    \end{equation}
    Since $(1,c) \in \SSS$ and $w(1) = 1$ one checks that 
    \[
    w_r \text{ contains } \SSS \setminus \{(r,1)\} \quad \Leftrightarrow \quad
    w \text{ contains } \SSS' = (\SSS \setminus \{(r,1),(1,c)\}) \cup \{(r,c)\}.
    \]
    Therefore $\AAA(x_\SSS)(w) = x_{\SSS'}(w)$.
\end{proof}

Lemma~\ref{lem:a-on-rook-placement} implies that $\AAA: \F[\symm_n(\RRRR)] \to \F[\overline{\symm_n(\RRRR)}]$ behaves well with respect to filtrations.

\begin{lemma}
    \label{lem:a-filtration}
    For any integer $d$, we have $\AAA( F_d \F[\symm_n(\RRRR)]) \subseteq F_d \F[\overline{\symm_n(\RRRR)}].$
\end{lemma}

\begin{proof}
    When $d < 0$ this reads $\AAA(0) \subseteq 0$, so assume $d \geq 0$. By Lemma~\ref{lem:board-filtration-span}, the $\F$-vector space 
    \[
    F_d \F[\symm_n(\RRRR)] = \text{image of $S_{\leq d}$ in } S/\II(\symm_n(\RRRR))
    \]
    is spanned by (the images of) the squarefree rook placement monomials 
    \begin{equation}
        \{ x_\SSS \,:\, \SSS \subseteq [n] \times [n] \text{ is a rook placement and } |\SSS| \leq d \}.
    \end{equation}
    Lemma~\ref{lem:a-on-rook-placement} implies that $\AAA(x_\SSS) \in F_{|\SSS|} \F[\overline{\symm_n(\RRRR)}]$ for any rook placement $\SSS \subseteq [n] \times [n]$ and the proof is complete.
\end{proof}

The reader may be tempted to prove Lemma~\ref{lem:a-filtration} by appealing to the graded action of $\symm_n \times \symm_n$ on $S$ by row and column permutation, thus avoiding the use of Lemma~\ref{lem:a-on-rook-placement}. This argument does not work because the permutation matrix sets $\symm_n(\RRRR)$ and $\overline{\symm_n(\RRRR)}$ and hence their vanishing ideals are not closed under the action of $\symm_n \times \symm_n$. 

The following result is another application of the averaging map $\AAA$ and Lemma~\ref{lem:a-on-rook-placement}. It will be play a crucial role in establishing the injectivity of $\gr \, \EEE: \RRR(\overline{\symm_n(\RRRR)}) \xrightarrow{ \, \, x_{1,1} \cdot (-) \, \, } \RRR(\symm_n(\RRRR)).$

\begin{lemma}
    \label{lem:filtration-intersect}
    For any integer $d$, we have \[\EEE(F_d \F[\overline{\symm_n(\RRRR)}]) = \EEE(\F[\overline{\symm_n(\RRRR)}]) \cap F_{d+1} \F[\symm_n(\RRRR)]\] as subspaces of $\F[\symm_n(\RRRR)]$.
\end{lemma}

\begin{proof}
    Our argument breaks into cases depending on the value of $d$. We handle these in ascending order of difficulty.

    {\bf Case 1:} $d < -1$.
    
    The desired equality reads \[\EEE(0) = \EEE(\F[\overline{\symm_n(\RRRR)}]) \cap 0,\]
    which is correct.

    {\bf Case 2:} $d = -1$.

    The left-hand side of the desired equality is
    \[
    \EEE(F_d \F[\overline{\symm_n(\RRRR)}]) = \EEE(F_{-1} \F[\overline{\symm_n(\RRRR)}]) = \EEE(0) = 0.
    \]
    We have
    \[
    F_{d+1} \F[\symm_n(\RRRR)] = F_0 \F[\symm_n(\RRRR)] = \{\text{all constant functions $\symm_n(\RRRR) \to \F$} \}.
    \]
    Since by assumption $n \geq 2$ and $\RRRR = \{(i,i) \,:\, i \in T\}$ for some subset $T \subseteq \{2,\dots,n\}$, there is at least one $w \in \symm_n(\RRRR)$ for which $w(1) \neq 1$ and $\overline{\symm_n(\RRRR)}$ is a {\bf proper} subset of $\symm_n(\RRRR)$. Since $\EEE$ acts on functions by extension by zero, there are no nonzero constant functions in $\EEE(\F[\overline{\symm_n(\RRRR)}])$ and
    \[
        \EEE(\F[\overline{\symm_n(\RRRR)}]) \cap F_{0} \F[\symm_n(\RRRR)] = 0
    \]
    so that $\EEE(F_d \F[\overline{\symm_n(\RRRR)}]) = \EEE(\F[\overline{\symm_n(\RRRR)}]) \cap F_{d+1} \F[\symm_n(\RRRR)]$ as subspaces of $\F[\symm_n(\RRRR)]$ in this case.

    {\bf Case 3:} $d \geq n-2$.

    Since $\symm_n(\RRRR) \subseteq \symm_n$, we have a graded surjection of quotient rings $\RRR(\symm_n) \twoheadrightarrow \RRR(\symm_n(\RRRR))$. Theorem~\ref{thm:permutation-matrix} implies that the top degree of $\RRR(\symm_n)$ is $n-1$, so the top degree of $\RRR(\symm_n(\RRRR))$ is at most $n-1$. Since $d \geq n-2$, Lemma~\ref{lem:associated-graded-isomorphism} implies
    \begin{equation}
    \label{eq:unbar-top}
        F_{d+1} \F[\symm_n(\RRRR)] = \F[\symm_n(\RRRR)]
    \end{equation}
    Since matrices in $\overline{\symm_n(\RRRR)}$ have identical first rows and columns, Lemma~\ref{lem:augment-slice} yields a graded surjection $\RRR(\symm_{n-1}) \twoheadrightarrow \RRR(\overline{\symm_n(\RRRR)})$ so that $\RRR(\overline{\symm_n(\RRRR)})$ has top degree at most $n-2$ by Theorem~\ref{thm:permutation-matrix}. Since $d \geq n-2$ we have
    \begin{equation}
    \label{eq:bar-top}
        F_d \F [\overline{\symm_n(\RRRR)}] = \F[\overline{\symm_n(\RRRR)}].
    \end{equation}
    Thanks to \eqref{eq:unbar-top} and \eqref{eq:bar-top}, the desired equality reads
    \begin{equation}
        \EEE( \F[\overline{\symm_n(\RRRR)}] = \EEE( \F[\overline{\symm_n(\RRRR)}] \cap \F[\symm_n(\RRRR)] \quad \text{as subspaces of $\F[\symm_n(\RRRR)]$}
    \end{equation}
    which is certainly true.

    {\bf Case 4:} $0 \leq d \leq n-3$.

    This case uses the averaging operator $\AAA$. Lemma~\ref{lem:extension-definition} gives the containment
    \begin{equation}
        \EEE(F_d \F[\overline{\symm_n(\RRRR)}]) \subseteq \EEE(\F[\overline{\symm_n(\RRRR)}]) \cap F_{d+1} \F[\symm_n(\RRRR)],
    \end{equation}
    so it suffices to prove the containment $\supseteq$. Let $f: \symm_n(\RRRR) \to \F$ be a function such that
    \begin{itemize}
        \item we have $f \in F_{d+1} \F[\symm_n(\RRRR)]$, and
        \item $f = \EEE(g)$ for some function $g: \overline{\symm_n(\RRRR)}\to \F$.
    \end{itemize}
    The function $g$ is the restriction of $f$ to $\overline{\symm_n(\RRRR)}$.
    We want to show $g \in F_d \F[\overline{\symm_n(\RRRR)}]$. 

    The function $f: \symm_n(\RRRR) \to \F$ is represented by a polynomial in $S$ of degree $\leq d+1$. 
    By Lemma~\ref{lem:board-filtration-span} we may write
    \[
        f = f_{d+1} + f_d + \cdots + f_1 + f_0
    \]
    where for $0 \leq i \leq d+1$ there exist scalars $\gamma_\SSS \in \F$ so that
    \begin{equation}
    \label{eq:f-to-xs}
        f_i = \sum_{|\SSS| = i} \gamma_\SSS \cdot x_\SSS \quad \text{as functions $\symm_n(\RRRR) \to \F$} \quad \quad (0 \leq i \leq d+1)
    \end{equation}
    where the sum ranges over rook placements $\SSS \subseteq [n] \times [n]$ with $i$ rooks. The expansions \eqref{eq:f-to-xs} are not unique.
    
    We have $f- f_{d+1} \in F_d \F[\symm_n(\RRRR)]$. Lemma~\ref{lem:a-filtration} implies
    \begin{equation}
    \label{eq:difference-filter}
        \AAA(f - f_{d+1}) \in F_d \F[\overline{\symm_n(\RRRR)}].
    \end{equation}
    Since $f = \EEE(g)$, Lemma~\ref{lem:ea-zero} yields
    \begin{equation}
    \label{eq:ae-vanish}
        \AAA(f) = \AAA(\EEE(g)) = 0.
    \end{equation}
    Since $\AAA$ is linear, \eqref{eq:difference-filter} and \eqref{eq:ae-vanish} imply
    \begin{equation}
        \label{eq:top-filter}
        \AAA(f_{d+1}) \in F_d \F[\overline{\symm_n(\RRRR)}].
    \end{equation}

    Consider the expansion \eqref{eq:f-to-xs} when $i = d+1$. We group the rook placements $\SSS \subseteq [n] \times [n]$ appearing on the right-hand side of \eqref{eq:f-to-xs} according to the five cases of Lemma~\ref{lem:a-on-rook-placement}. 
    In particular, set 
    \[
        T := \{ \text{rook placements } \SSS \subseteq [n] \times [n] \,:\, \SSS \text{ has $d+1$ rooks} \}.
    \]
    We have a disjoint union decomposition
    \[
    T = T_1 \sqcup T_2 \sqcup T_3 \sqcup T_4 \sqcup T_5
    \]
    where
     \begin{align*}
        T_1 &:= \{ \SSS \in T \,:\,  \SSS \text{ does not meet the first row or column} \}, \\
        T_2 &:= \{ \SSS \in T \,:\, \SSS \text{ contains $(1,1)$} \}, \\
        T_3 &:= \{ \SSS \in T \,:\, \SSS \text{ meets the first row but not the first column} \}, \\
        T_4 &:= \{ \SSS \in T \,:\, \SSS \text{ meets the first column but not the first row} \}, \\
        T_5 &:= \{ \SSS \in T \,:\, \SSS \text{ meets both the first row and the first column but does not contain $(1,1)$} \}.
    \end{align*}
    We have
    \begin{equation}
        \label{eq:f-to-h}
        f_{d+1} = h_1 + h_2 + h_3 +h_4 + h_5 \quad \text{as functions $\symm_n(\RRRR) \to \F$}
    \end{equation}
    where for $1 \leq i \leq 5$, the function $h_i: \symm_n(\RRRR) \to \F$
    is given by
    \begin{equation}
        \label{eq:formula-for-h}
        h_i := \sum_{\SSS \in T_i} \gamma_{\SSS} \cdot x_{\SSS}
    \end{equation}
    Lemma~\ref{lem:a-on-rook-placement} implies 
    \begin{equation}
        \label{eq:a-on-1}
        \AAA(h_1) = (n - 2 - d) \cdot h_1 \quad \text{as functions $\overline{\symm_n(\RRRR)} \to \F$}
    \end{equation}
    and
    \begin{equation}
    \label{eq:a-on-25}
        \AAA(h_2 + h_3 + h_4 + h_5) = \AAA(h_2) + \AAA(h_3) + \AAA(h_4) + \AAA(h_5) \in F_d \F[\overline{\symm_n(\RRRR)}].
    \end{equation}
    Combining \eqref{eq:top-filter} and \eqref{eq:a-on-25} implies $\AAA(h_1) \in F_d \F[\overline{\symm_n(\RRRR)}]$. Since $0 \leq d \leq n-3$, we have $n-2-d \neq 0$ and \eqref{eq:a-on-1} implies 
    \begin{equation}
        \label{eq:1-filtration}
        h_1 \in F_d \F[\overline{\symm_n(\RRRR)}].
    \end{equation}
    The functions $h_3,h_4,h_5: \symm_n(\RRRR) \to \F$ are linear combinations of $x_\SSS$ for rook placements $\SSS$ involving rooks in the first row and column $\neq (1,1)$. The functions $h_3, h_4, h_5$ therefore vanish upon restriction to $\overline{\symm_n(\RRRR)}$.
    We have the identity
    \begin{align}
        g  &= (f - f_{d+1}) + f_{d+1}  &\text{as functions $\overline{\symm_n(\RRRR)} \to \F$} \\
        &= (f - f_{d+1}) + h_1 + h_2 + h_3 + h_4 + h_5  &\text{as functions $\overline{\symm_n(\RRRR)} \to \F$} \\
        &= (f - f_{d+1}) + h_1 + h_2 &\text{as functions $\overline{\symm_n(\RRRR)} \to \F$}
    \end{align}
    of functions $\overline{\symm_n(\RRRR)} \to \F.$ 
    Since $x_{1,1}$ is the constant function 1 on $\overline{\symm_n(\RRRR)}$, we have
    \begin{equation}
        h_2 = \sum_{\substack{|\SSS| = d+1 \\ (1,1) \in \SSS}} \gamma_\SSS \cdot x_\SSS = \sum_{\substack{|\SSS| = d+1 \\ (1,1) \in \SSS}}  \gamma_\SSS \cdot x_{\SSS \setminus(1,1)} \quad 
        \text{as functions $\overline{\symm_n(\RRRR)} \to \F$}
    \end{equation}
    so $h_2 \in F_d \F[\overline{\symm_n(\RRRR)}]$.
    Since $f - f_{d+1} \in F_d \F[\overline{\symm_n(\RRRR)}]$, the membership \eqref{eq:1-filtration} implies that $g \in F_d \F[\overline{\symm_n(\RRRR)}]$ and the proof is complete.
\end{proof}

The assumption that $\F$ has characteristic zero was used for the first time in the above proof to yield $n-2-d \neq 0$ for $0 \leq d \leq n-3$.
We use Lemma~\ref{lem:filtration-intersect} to prove that the associated graded of the extension map $\EEE: \F[\overline{\symm_n(\RRRR)}] \to \F[\symm_n(\RRRR)]$ is injective. The injection $\gr \, \EEE$ the first part of our deletion-contraction short exact sequence.

\begin{lemma}
    \label{lem:gr-e-injective}
    The $S$-module homomorphism $\gr \, \EEE: \RRR(\overline{\symm_n(\RRRR)}) \xrightarrow{\, \, x_{1,1} \cdot(-) \, \, } \RRR(\symm_n(\RRRR))$
    induced by multiplication by $x_{1,1}$ is injective.
\end{lemma}

\begin{proof}
    Lemma~\ref{lem:associated-graded-isomorphism} and Lemma~\ref{lem:extension-definition} identify $\gr \, \EEE$ with the morphism
    \begin{equation}
        \gr \, \EEE: \gr \, \F[\overline{\symm_n(\RRRR)}] \longrightarrow \gr \, \F[\symm_n(\RRRR)]
    \end{equation}
    where the filtrations on $\F[\overline{\symm_n(\RRRR)}]$ and $\F[\symm_n(\RRRR)]$ are those of Lemma~\ref{lem:filtration-intersect}. Recalling that $\gr \, \EEE$ is homogeneous of degree 1, we have
    \begin{equation}
        \ker( \gr \, \EEE) = \bigoplus_{d \geq 0} \ker(\gr \, \EEE)_d
    \end{equation}
    where
    \begin{equation}
        \ker(\gr \, \EEE)_d = \{ f + F_{d-1} \F[\overline{\symm_n(\RRRR)}] \,:\, f \in F_d \F[\overline{\symm_n(\RRRR)}] \text{ and } \EEE(f) \in F_d \F[\symm_n(\RRRR)] \}.
    \end{equation}
    Given $f \in F_d \F[\overline{\symm_n(\RRRR)}]$ so that $f + F_{d-1}\F[\overline{\symm_n(\RRRR)}] \in \ker(\gr \, \EEE)_d$, by Lemma~\ref{lem:filtration-intersect} there exists $g \in F_{d-1} \F[\overline{\symm_n(\RRRR)}]$ so that $\EEE(f) = \EEE(g)$. But $\EEE: \F[\overline{\symm_n(\RRRR)}] \to \F[\symm_n(\RRRR)]$ is injective, so we have $f=g$ and 
    \[
    f + F_{d-1}\F[\overline{\symm_n(\RRRR)}] = g + F_{d-1}\F[\overline{\symm_n(\RRRR)}] = 0 \quad \text{in } \gr \, \F[\overline{\symm_n(\RRRR)}].
    \]
    We conclude that $\ker(\gr \, \EEE) = 0$ and that $\gr \, \EEE$ is injective.
\end{proof}

The last result of this section phrases Lemma~\ref{lem:gr-e-injective} in a more convenient form. Let 
\begin{equation}
    \overline{S} := \F[x_{i,j} \,:\, 2 \leq i,j \leq n]
\end{equation}
be the subring of $S$ formed by setting the variables in the first row and column to zero and regard $\overline{S}$ as the coordinate ring of the affine space $\F^{(n-1) \times (n-1)}$ of matrices with rows and columns indexed by $\{2,\dots,n\} \times \{2,\dots,n\}$. Since $\RRRR \subseteq \{2,\dots,n\}\times\{2,\dots,n\}$, we have a locus $\symm_{n-1}(\RRRR) \subseteq \F^{(n-1) \times (n-1)}$ and a graded orbit harmonics quotient $\RRR(\symm_{n-1}(\RRRR))$ of $\overline{S}$. Multiplication by the variable $x_{1,1}$ gives an $\overline{S}$-module homomorphism $x_{1,1} \cdot (-): \overline{S} \to S$.

\begin{lemma}
    \label{lem:x11-injection}
    Let $n \geq 2$ and $\RRRR \subseteq \{ (2,2), (3,3) , \dots, (n,n) \}$. Multiplication by $x_{1,1}$ induces an injection of $\overline{S}$-modules
    \[
    \RRR(\symm_{n-1}(\RRRR)) \xrightarrow{\, \, x_{1,1} \cdot (-) \, \, } \RRR(\symm_n(\RRRR)).
    \]
\end{lemma}

\begin{proof}
    The $n \times n$ matrices in the locus $\overline{\symm_n(\RRRR)}$ are obtained from the $(n-1) \times (n-1)$ in $\symm_{n-1}(\RRRR)$ by appending the map
    \[
    A \mapsto (1) \oplus A \quad \text{(matrix direct sum)}.
    \]
    Lemma~\ref{lem:augment-slice} gives an isomorphism $\RRR(\symm_{n-1}(\RRRR)) \xrightarrow{ \, \sim \, } \RRR(\overline{\symm_{n}(\RRRR)})$ induced by the inclusion $\overline{S} \hookrightarrow S$. We have the composite map $\RRR(\symm_{n-1}(\RRRR)) \to \RRR(\overline{\symm_n(\RRRR)}) \to \RRR(\symm_n(\RRRR))$ given by
    \[
    f + \gr \, \II(\symm_{n-1}(\RRRR)) \xmapsto{\, \sim \, } f + \gr \, \II(\overline{\symm_n(\RRRR)}) \xmapsto{  \, \, x_{1,1} \cdot (-) \, \, } x_{1,1} \cdot f + \gr \, \II(\symm_n(\RRRR)).
    \]
    where the second map is the injection of Lemma~\ref{lem:gr-e-injective}  induced by multiplication by $x_{1,1}$.
\end{proof}

\section{Hilbert series and presentation}
\label{sec:Hilbert}

Let $\RRRR \subseteq [n] \times [n]$ be a rook placement. The first two goals of this section are to give a generating set of the defining ideal $\gr \, \II(\symm_n(\RRRR))$ of $\RRR(\symm_n(\RRRR))$ and exhibit a deletion-contraction short exact sequence relating the rings $\RRR(\symm_n(\RRRR))$ for varying $\RRRR$. These goals will be achieved simultaneously using an inductive argument.

\subsection{Ideal generators and a deletion-contraction short exact sequence} 
Our starting point is a lower bound for the orbit harmonics ideal $\gr \, \II(\symm_n(\RRRR))$.

\begin{lemma}
    \label{lem:ideal-containment}
    Let $\RRRR \subseteq [n] \times [n]$ be a rook placement. We have
    \begin{equation}
        I_n + (x_{i,j} \,:\, (i,j) \in \RRRR) \subseteq \gr \, \II(\symm_n(\RRRR))
    \end{equation}
    as ideals in $S$.
\end{lemma} 

\begin{proof}
    Since $\symm_n(\RRRR) \subseteq \symm_n$, Theorem~\ref{thm:permutation-matrix} implies $I_n = \gr \, \II(\symm_n) \subseteq \gr \, \II(\symm_n(\RRRR))$. Furthermore, for any $(i,j) \in \RRRR$ we have $x_{i,j} \in \II(\symm_n(\RRRR))$ so that $(i,j) \in \gr \, \II(\symm_n(\RRRR))$.
\end{proof}

The containment of ideals in Lemma~\ref{lem:ideal-containment} is not always an equality. For example, if $n  = 1$ and $\RRRR = \{(1,1)\}$ we have $\RRR(\symm_1(\RRRR)) = \RRR(\DDD_1) = 0$ so that $\gr \, \II(\symm_n(\RRRR)) = \gr \, \II(\DDD_1) = S$ whereas $I_1 + (x_{1,1}) = (x_{1,1}) \subsetneq S$. Our first main result implies that this is  the only case where the containment in Lemma~\ref{lem:ideal-containment} is not an equality. For $1 \leq i_0, j_0 \leq n$, we write
\begin{equation}
    S^{(i_0,j_0)} := \F[x_{i,j} \,:\, 1 \leq i,j \leq n, \, \, i \neq i_0 \text{ and } j \neq j_0]
\end{equation}
for the subring of $S$ obtained by removing the variables in row $i_0$ and column $j_0$. If $\RRRR$ is a rook placement on the board $[n-1] \times [n-1]$, let $\RRR(\symm_{n-1}(\RRRR))^{(i_0,j_0)}$ be the orbit harmonics quotient of the ring $S^{(i_0,j_0)}$ obtained by regarding $\RRRR$ as a rook placement on $\{1 ,\dots, \widehat{i_0}, \dots, n\} \times \{1, \dots, \widehat{j_0} , \dots, n \}$.

\begin{theorem}
    \label{thm:ses-presentation}
    Let $\RRRR \subseteq [n] \times [n]$ be a rook placement.
    \begin{enumerate}
        \item 
        Suppose $n \geq 2$, the rook placement $\RRRR$ is nonempty, and let $(i_0,j_0) \in \RRRR$ be a rook.
        Let $\RRRR \setminus (i_0,j_0)$ be the deletion and $\RRRR / (i_0,j_0)$ be the contraction.
        We have a short exact sequence of $S^{(i_0,j_0)}$-modules 
        \[
        0 \to \RRR(\symm_{n-1}(\RRRR / (i_0,j_0)))^{(i_0,j_0)} \xrightarrow {\, \, \varphi \, \, } \RRR(\symm_n(\RRRR \setminus (i_0,j_0))) \xrightarrow{\, \, \pi \, \, } \RRR(\symm_n(\RRRR)) \to 0.
        \]
        Here $\varphi: \RRR(\symm_{n-1}(\RRRR / (i_0,j_0)))^{(i_0,j_0)} \to \RRR(\symm_n(\RRRR \setminus (i_0,j_0)))$ is induced by multiplication by $x_{i_0,j_0}$ and $\pi: \RRR(\symm_n(\RRRR \setminus (i_0,j_0))) \to \RRR(\symm_n(\RRRR))$ is the canonical projection.
        \item For $n \neq 1$ or when $\RRRR = \varnothing$, the ideal $\gr \, \II(\symm_n(\RRRR)) \subseteq S$ is given by  
        \[
        \gr \, \II(\symm_n(\RRRR)) = I_n + (x_{i,j} \,:\, (i,j) \in \RRRR).
        \]
        When $n = 1$ and $\RRRR = \{ (1,1) \}$ we have $\gr \, \II(\symm_1(\RRRR)) = \gr \, \II(\DDD_1) = S$.
    \end{enumerate}
\end{theorem}

\begin{proof}
    We establish both (1) and (2) simultaneously by induction on $n$ and $|\RRRR|$. The relevant base cases are determined by the hypotheses of Lemma~\ref{lem:x11-injection}.
    
    When $\RRRR = \varnothing$, the ideal presentation (2) holds by Theorem~\ref{thm:permutation-matrix}. When $n = 1$ and $\RRRR = \{(1,1)\}$ we have $\symm_1(\RRRR) = \DDD_1 = \varnothing$ so that $\gr \, \II(\symm_1(\RRRR)) = S$ and (2) holds, as well. When $n = 2$, one quickly checks that
    \begin{equation}
        \dim_\F S/(I_2 + (x_{i,j} \,:\, (i,j) \in \RRRR)) = \begin{cases}
            2 & \RRRR = \varnothing \\
            1 & \RRRR \neq \varnothing
        \end{cases}
    \end{equation}
    so that $\dim_\F S/(I_2 + (x_{i,j} \,:\, (i,j) \in \RRRR)) = |\symm_2(\RRRR)|$ in either case and (2) holds  by Lemma~\ref{lem:ideal-containment}.
    When $n = 2$ and $|\RRRR| = 1$, the sequence of maps in (1) is isomorphic to
    \begin{equation}
        0 \to \F \xrightarrow{ \, \, x \cdot (-) \, \, } \F[x]/(x^2) \xrightarrow{ \, \, \mathrm{can.} \, \, } \F[x]/(x) \to 0
    \end{equation}
    which is exact. When $n = 2$ and $|\RRRR| = 2$, the sequence of maps in (1) is 
    \begin{equation}
        0 \to 0 \to \F \xrightarrow{\, \, \mathrm{id.} \, \, } \F \to 0
    \end{equation}
    which is also exact.

    Let $n \geq 3$ and let $\RRRR \subseteq [n[ \times [n]$ be a nonempty rook placement with $(i_0,j_0) \in \RRRR$. Inductively assume that (2) holds for all $(n',\RRRR')$ with $n' < n$ and all $(n,\RRRR'')$ with $|\RRRR''| < |\RRRR|$. In particular, since $n \geq 3$ we have 
    \begin{equation}
    \label{eq:contraction-presentation}
        \RRR(\symm_{n-1}(\RRRR / (i_0,j_0)))^{(i_0,j_0)} = S^{(i_0,j_0)}/(I_{n-1}^{(i_0,j_0)} + (x_{i,j} \,:\, (i,j) \in \RRRR / (i_0,j_0))
    \end{equation}
    and also
    \begin{equation}
    \label{eq:deletion-presentation}
        \RRR(\symm_{n}(\RRRR \setminus (i_0,j_0))) = S/(I_n + (x_{i,j} \,:\, (i,j) \in \RRRR \setminus (i_0,j_0)).
    \end{equation}
     Lemma~\ref{lem:x11-injection} yields an injection of $S^{(i_0,j_0)}$-modules 
    \begin{equation}
        \varphi: \RRR(\symm_{n-1}(\RRRR / (i_0,j_0)))^{(i_0,j_0)} \xrightarrow{\, \, x_{i_0,j_0} \cdot (-) \, \, } \RRR(\symm_n(\RRRR \setminus (i_0,j_0)))
    \end{equation}
    induced by multiplication by $x_{i_0,j_0}$. The containment of loci $\symm_n(\RRRR) \subseteq \symm_n(\RRRR \setminus(i_0,j_0))$ gives rise to a canonical surjection
    \begin{equation}
        \pi: \RRR(\symm_n(\RRRR \setminus(i_0,j_0))) \xrightarrow{ \, \, \mathrm{can.} \, \, } \RRR(\symm_n(\RRRR)).
    \end{equation}
    Lemma~\ref{lem:ideal-containment} gives the containment of ideals
    \begin{equation}
        I_n + (x_{i,j} \,:\, (i,j) \in \RRRR) \subseteq \gr \, \II(\symm_n(\RRRR)).
    \end{equation}
    Since $(i_0,j_0) \in \RRRR$ we have $x_{i_0,j_0} \in \ker(\pi)$ so that $\pi \circ \varphi = 0$ and we have a complex 
    \begin{equation}
    \label{eq:goal-ses}
        0 \to \RRR(\symm_{n-1}(\RRRR / (i_0,j_0)))^{(i_0,j_0)} \xrightarrow {\, \, \varphi \, \, } \RRR(\symm_n(\RRRR \setminus (i_0,j_0))) \xrightarrow{\, \, \pi \, \, } \RRR(\symm_n(\RRRR)) \to 0
    \end{equation}
    which is exact on the left and on the right. Since 
    $\dim_\F \RRR(\Zpoints) = |\Zpoints|$ for any locus $\Zpoints$ and 
    \begin{equation}
        | \symm_n( \RRRR \setminus(i_0,j_0)))| = | \symm_{n-1}(\RRRR / (i_0,j_0)))| + |\symm_n(\RRRR)|,
    \end{equation}
    the complex~\eqref{eq:goal-ses} is in fact a short exact sequence and (1) is established for $n, \RRRR,$ and $(i_0,j_0)$. Since \eqref{eq:goal-ses} is exact and $\varphi$ is induced by multiplication by $x_{i_0,j_0}$, the defining ideal of $\RRR(\symm_n(\RRRR))$ is given by 
    \begin{equation}
        \gr \, \II(\symm_n(\RRRR)) = \gr \, \II(\symm_n(\RRRR \setminus(i_0,j_0))) + (x_{i_0,j_0}) = I_n + (x_{i,j} \,:\, (i,j) \in \RRRR)
    \end{equation}
    where the second equality used \eqref{eq:deletion-presentation}. This proves (2) for $(n,\RRRR)$.
\end{proof}

For $\BBB \subseteq [n] \times [n]$, recall that $\symm_n(\BBB)$ is the locus of $n \times n$ permutation matrices which avoid $\BBB$. We pose the following problem.

\begin{problem}
    \label{prob:forbidden}
    For $\BBB \subseteq [n] \times [n]$, give an explicit set of generators of the defining ideal $\gr \, \II(\symm_n(\BBB))$ of $\RRR(\symm_n(\BBB))$.
\end{problem}

The argument of Lemma~\ref{lem:ideal-containment} shows that $I_n + (x_{i,j} \,:\, (i,j) \in \BBB) \subseteq \gr \, \II(\symm_n(\BBB))$, but this containment is often proper. For example, when $n = 3$ and $\BBB = \{(1,2),(1,3)\}$ one has $|\symm_3(\BBB)| = 2$ but
$\dim_\F S/(I_3 + (x_{1,2}, x_{1,3})) = 3.$ 

A conjectural solution to Problem~\ref{prob:forbidden} is as follows. Given $\BBB \subseteq [n] \times [n]$, call a subset $\III \subseteq [n] \times [n]$ {\em $\BBB$-inadmissible} if there does not exist $w \in \symm_n$ such that
\[
\III \subseteq \{(i,w(i) \,:\, 1 \leq i \leq n \} \subseteq ([n] \times [n]) \setminus \BBB.
\]
We have the monomial $x_\III = \prod_{(i,j) \in \III} x_{i,j}$. Define an ideal $J_\BBB \subseteq S$ by
\begin{equation}
    J_\BBB := I_n + (x_{\III} \,:\, \III \subseteq [n] \times [n] \text{ is $\BBB$-inadmissible}).
\end{equation}
In particular, if $(i,j) \notin \BBB$ then $x_{i,j}$ is a generator of $J_\BBB$. Also, if there does not exist $w \in \symm_n$ whose graph does not meet $\BBB$, then $\III = \varnothing$ is inadmissible and $J_\BBB = S$. Limited computational evidence suggests that
\begin{equation}
\label{eq:conjectural-forbidden-presentation}
    J_\BBB = \gr \, \II(\symm_n(\BBB)) \quad \quad \text{as ideals in $S$.}
\end{equation}
Since $x_\III \in \II(\symm_n(\BBB))$ for any inadmissible $\III \subseteq [n] \times [n]$, we have $J_\BBB \subseteq \gr \, \II(\symm_n(\BBB))$.

\subsection{Foata and Hilbert}
Recall that the Foata transformation $\Psi: \symm_n \xrightarrow{\,\, \sim \, \, } \symm_n$ is (in our somewhat nonstandard convention) given by letting $\Psi(w) \in \symm_n$ be the permutation obtained from $w$ by the following procedure.
\begin{enumerate}
    \item Write the cycles of $w$ starting with their smallest elements, and in decreasing order of smallest elements from left-to-right.
    \item Erase the parentheses and interpret the result as the one-line notation of $\Psi(w) \in \symm_n.$
\end{enumerate}
For example, we have \[\Psi: (4,7,5) (2)(1,3,6) \mapsto [4,7,5,2,1,3,6].\] We give a simple result on how the Foata transformation interacts with fixed points of permutations.

\begin{lemma}
    \label{lem:foata-restrict}
    Let $1 \leq m \leq n$ with $n > 1$ and suppose $w \in \symm_n$ satisfies $w(m) = m$. We have 
    \[
    \lis(\Psi(w)) = \lis(\overline{\Psi(w)})
    \]
    where the length $n-1$ sequence $\overline{\Psi(w)}$ is obtained from the length $n$ sequence $\Psi(w)$ by erasing $m$ and $\lis(\overline{\Psi(w)})$ is the length of the longest increasing subsequence of $\overline{\Psi(w)}$.
\end{lemma}

In the example above with $w= (4,7,5) (2)(1,3,6) \in \symm_7$ and $m= 2$, we have 
\[
\Psi(w) = [4,7,5,2,1,3,6] \quad \text{and} \quad \overline{\Psi(w)} = [4,7,5,1,3,6]
\]
with
\[
\lis(\Psi(w)) = \lis(\overline{\Psi(w)}) = 3.
\]

\begin{proof}
    Write $\Psi(w) = [a_1,\dots,a_n] \in \symm_n$ and suppose $a_i = m$ so that $\overline{\Psi(w)} = [a_1,\dots,\widehat{a_i},\dots,a_n].$
    
    Suppose first that $i < n$. Since $w(m) =m$, both $a_i$ and $a_{i+1}$ are left-to-right minima in the word $[a_1,\dots,a_n]$ with $a_i > a_{i+1}$. Since $a_i$ is a left-to-right minimum, and increasing subsequence of $[a_1,\dots,a_n]$ involving $a_i$ must start with $a_i$, and since $a_i > a_{i+1}$ any such increasing subsequence can be transformed into a new increasing subsequence of the same length by replacing $a_i$ with $a_{i+1}$. In summary, we have
    \begin{equation}
    \lis( [a_1,\dots,a_n]) = \lis( [a_1 ,\dots, \widehat{a_i}, \dots, a_n]).
    \end{equation}
    
    Now suppose $i=n$. The definition of $\Psi$ and the fact that $w(m)= m$ forces $m=1$ and we certainly have
    \begin{equation}
    \lis( [a_1,\dots,a_n]) = \lis([ a_1,\dots,a_{n-1},1]) = \lis([a_1,\dots,a_{n-1}])
    \end{equation}
    which completes the proof.
\end{proof}

We apply the combinatorial Lemma~\ref{lem:foata-restrict} together with the short exact sequence in Theorem~\ref{thm:ses-presentation} (1) to derive the Hilbert series of $\RRR(\DDD_n)$.

\begin{theorem}
\label{thm:hilbert}
    The Hilbert series of $\RRR(\DDD_n)$ is given by
    \[
    \Hilb(\RRR(\DDD_n);q) = \sum_{w \in \DDD_n} q^{n - \lis(\Psi(w))}
    \]
    where $\Psi: \symm_n \xrightarrow{\,\, \sim \, \, } \symm_n$ is the Foata transformation.
\end{theorem}

\begin{proof}
    The theorem is easily checked for $n=0,1$, so we assume $n > 1$.
    Let $\symm_{n,m} \subseteq \symm_n$ be the set of permutations
\begin{equation}
\symm_{n,m} := \{ w \in \symm_n \,:\, w(i) \neq i \text{ for } i > m \}.
\end{equation}
In particular, we have $\symm_{n,n} = \symm_n$ and $\symm_{n,0} = \DDD_n$. We have $\symm_{n,m} = \symm_n(\RRRR)$ where \[\RRRR = \{(i,i) \,:\, m < i \leq n \}.\]
Theorem~\ref{thm:permutation-matrix} implies that for $m=n$ one has 
\begin{equation}
\label{eq:hilbert-restatement}
    \Hilb(\RRR(\symm_{n,n});q) = \Hilb(\RRR(\symm_n);q) = \sum_{w \in \symm_n} q^{n-\lis(w)}.
\end{equation}

Suppose $0 \leq m < n$. Theorem~\ref{thm:ses-presentation} (1) gives a short exact sequence
\begin{equation}
\label{eq:reformulated-ses}
    0 \to \RRR(\symm_{n-1,m}) \to \RRR(\symm_{n,m+1}) \to \RRR(\symm_{n,m}) \to 0
\end{equation}
where the first map is homogeneous of degree 1  and the second map is homogeneous of degree 0. 
We have the chain of equalities
\begin{align}
    \Hilb(\RRR(\symm_{n,m});q) &= \Hilb(\RRR(\symm_{n,m+1});q) - q \cdot \Hilb(\RRR(\symm_{n-1,m});q) \\
    &= \sum_{w \in \symm_{n,m+1}} q^{n - \lis(\Psi(w))} - q \cdot \sum_{w \in \symm_{n-1,m}} q^{n-1-\lis(\Psi(w))} \\
    &= \sum_{w \in \symm_{n,m+1}} q^{n - \lis(\Psi(w))} -   \sum_{w \in \symm_{n-1,m}} q^{n-\lis(\Psi(w))} \\
    &= \sum_{w \in \symm_{n,m+1}} q^{n - \lis(\Psi(w))} -   \sum_{\substack{w \in \symm_{n,m+1} \\ w(m+1) = m+1}} q^{n-\lis(\Psi(w))} \\
    &= \sum_{w \in \symm_{n,m}} q^{n - \lis(\Psi(w))}.
\end{align}
The first equality records the exactness and degree properties of  \eqref{eq:reformulated-ses}.
The second equality is by upward induction on $n$ and downward induction on $m$ (with the base case on $m$ handled by Equation~\eqref{eq:hilbert-restatement} and the fact that $\Psi: \symm_{n-1} \to \symm_{n-1}$ is a bijection).
The third equality is absorption.
The fourth equality is Lemma~\ref{lem:foata-restrict}.
The fifth equality is subtraction.
Taking $m=0$ completes the proof.
\end{proof}

The above proof and $(\symm_n \times \symm_n)$-equivariance give the stronger statement
\begin{equation}
\label{eq:general-rook-hilbert}
    \Hilb(\RRR(\symm_n(\RRRR));q) = \sum_{w \in \symm_{n,m}} q^{n-\lis(\Psi(w))}
\end{equation}
whenever $\RRRR \subseteq [n] \times [n]$ is a rook placement with $|\RRRR| = n-m$.
When $n = 4$, one has the following data
\[
\begin{tabular}{c | c | c | c}

$w \in \DDD_n$ & $\Psi(w)$ & $\lis(\Psi(w))$ & $n - \lis(\Psi(w))$ \\
\hline 
$(1,2,3,4)$ & $[1,2,3,4]$ & $4$ & $0$ \\
$(1,3,2,4)$ & $[1,3,2,4]$ & $3$ & $1$ \\
$(1,2,4,3)$ & $[1,2,4,3]$ & $3$ & $1$ \\
$(1,3,4,2)$ & $[1,3,4,2]$ & $3$ & $1$ \\
$(1,4,2,3)$ & $[1,4,2,3]$ & $3$ & $1$ \\
$(1,4,3,2)$ & $[1,4,3,2]$ & $2$ & $2$ \\
$(3,4)(1,2)$ & $[3,4,1,2]$ & $2$ & $2$ \\
$(2,4)(1,3)$ & $[2,4,1,3]$ & $2$ & $2$ \\
$(2,3)(1,4)$ & $[2,3,1,4]$ & $3$ & $1$ \\
\end{tabular}
\]
so that Theorem~\ref{thm:hilbert} gives
\begin{equation*}
    \Hilb(\RRR(\DDD_n);q) = 1 + 5 q + 3 q^2.
\end{equation*}

\section{Module structure}
\label{sec:Module}

In this section we calculate the graded $\symm_n$-module structure of $\RRR(\DDD_n)$. As with the Hilbert series of $\RRR(\DDD_n)$, this will be achieved by considering related quotient rings over submatrices of $\xx_{n \times n}$, relating these quotient rings via exact sequences, and applying induction. The short exact sequence in Theorem~\ref{thm:ses-presentation} (1) and the mapping cone result of Theorem~\ref{thm:exact-cone} will play major roles. We begin by setting up our notation.

\subsection{Some algebraic notation} Let $K \subseteq [n]$. We define
\begin{equation}
    S^{(K)} := \F[x_{i,j} \,:\, 1 \leq i,j \leq n, \, \, i, j \notin K] \subseteq S
\end{equation}
to be the subalgebra of $S$ generated by variables $x_{i,j}$ for which $i,j \notin K$. The ring $S^{(K)}$ carries a graded action of the subgroup $\symm_{[n] - K} \subseteq \symm_n$ by $w \cdot x_{i,j} := x_{w(i),w(j)}$.
We write $I_{n-k}^{(K)} \subseteq S^{(K)}$ for the ideal with the same generating set as in Definition~\ref{def:I-ideal-definition}, but over the $(n-k) \times (n-k)$ matrix of variables generating $S^{(K)}$. We write
\begin{equation}
    \RRR(\symm_{n-k})^{(K)} := S^{(K)}/I^{(K)}_{n-k}
\end{equation}
for the corresponding copy of the quotient ring $\RRR(\symm_{n-k})$. The ring $\RRR(\symm_{n-k})^{(K)}$ carries a graded action of $\symm_{[n]-K}$. 

We will need to relate the quotient rings $\RRR(\symm_{n-k})^{(K)}$ to one another. This will be done using the following maps.

\begin{lemma}
    \label{lem:K-maps}
    Suppose $K \subseteq [n]$ satisfies $|K| = k$ and let $i \in K$. Multiplication by $x_{i,i}$ induces a $S^{(K)}$-module homomorphism
    \[
   \varphi:  \RRR(\symm_{n-k})^{(K)} \longrightarrow \RRR(\symm_{n-k+1})^{(K-i)}
    \]
    given by $\varphi: f + I_{n-k}^{(K)} \mapsto x_{i,i} \cdot f + I_{n-k+1}^{(K-i)}$. The homomorphism $\varphi$ is homogeneous of degree 1.
\end{lemma}

\begin{proof}
    One checks on generators that $I_{n-k}^{(K)}$ is annihilated by the composition
    \begin{equation}
        S^{(K)} \xrightarrow{ \, \, \times x_{i,i} \, \, } S^{(K-i)} \xrightarrow{\, \, \mathrm{can. \, \, } } \RRR(\symm_{n-k+1})^{(K-i)}
    \end{equation}
    where the first map is multiplication by $x_{i,i}$ and the second is the canonical projection.
\end{proof}

Let $U \cong \F^n$ be the $n$-dimensional $\F$-vector space with distinguished basis
\begin{equation}
    U := \F \cdot \{ \ee_1, \dots, \ee_n \}.
\end{equation}
Let $\bigwedge U = \bigoplus_{k=0}^n \bigwedge^k U$ be the free exterior algebra over $U$. The vector space $\bigwedge^k U$ has basis $\{ \ee_K \,:\, K \subseteq [n], \,\, |K|=k \}$ where 
\begin{equation}
    \ee_K := \ee_{i_1} \wedge \cdots \wedge \ee_{i_k} \quad \text{for } K = \{i_1 < \cdots < i_k \}.
\end{equation}

We put a graded structure on the tensor product
\begin{equation}
    S \otimes \big(\bigwedge U\big) \cong \bigoplus_{K \subseteq [n]} S \cdot \ee_K
\end{equation}
by setting 
\begin{equation}
    \deg(x_{i,j}) = 1 \quad \text{and} \quad \deg(\ee_i) = 0
\end{equation}
for all $1 \leq i,j \leq n$. The rules 
\begin{equation}
    w \cdot x_{i,j} := x_{w(i),w(j)} \quad \text{and} \quad w \cdot \ee_i := \ee_{w(i)} \quad \quad \text{($w \in \symm_n$,  $ \, \, 1 \leq i,j \leq n$)}
\end{equation}
give $S \otimes (\bigwedge U)$ the structure of a graded $\symm_n$-module. The action of $\symm_n$ on the given direct sum decomposition of $S \otimes (\bigwedge U)$ is determined by 
\begin{equation}
    w \cdot (S \cdot \ee_K) = S \cdot \ee_{w(K)} \quad \quad (w \in \symm_n, \, \, K \subseteq [n])
\end{equation}
where $w(K) = \{w(k) \,:\, k \in K \}$.

For any $0 \leq k \leq n$, consider the linear subspace of $S \otimes (\bigwedge U)$ given by 
\begin{equation}
    \bigoplus_{|K|=k} S^{(K)} \cdot \ee_K \subseteq S \otimes \big(\bigwedge U\big)
\end{equation}
where the direct sum is over $k$-element subsets $K \subseteq [n]$. One can check that $\bigoplus_{|K|=k} S^{(K)} \cdot \ee_K$ is a graded subspace of $S \otimes (\bigwedge U)$ which is closed under the action of $\symm_n$. We have a further graded subspace
\begin{equation}
    \bigoplus_{|K|=k} I_{n-k}^{(K)} \cdot \ee_K \subseteq
    \bigoplus_{|K|=k} S^{(K)} \cdot \ee_K \subseteq S \otimes \big(\bigwedge U\big)
\end{equation}
and the generating set of Definition~\ref{def:I-ideal-definition} implies that $\bigoplus_{|K|=k} I_{n-k}^{(K)} \cdot \ee_K$ is also $\symm_n$-stable. We write
\begin{equation}
\label{eq:v-definition}
    V_k := \frac{\bigoplus_{|K|=k} S^{(K)} \cdot \ee_K}{\bigoplus_{|K|=k} I^{(K)}_{n-k} \cdot \ee_K} \cong \bigoplus_{|K| =k} \RRR(\symm_{n-k})^{(K)} \cdot \ee_K
\end{equation}
for the corresponding graded quotient module. The graded $\symm_n$-structure of $V_k$ is as follows.

\begin{lemma}
    \label{lem:v-induction}
    We have an identification of graded $\symm_n$-modules 
    \begin{equation}
        V_k \cong \Ind_{\symm_{n-k} \times \symm_k}^{\symm_n} (\RRR(\symm_{n-k}) \otimes_\F \sign_{\symm_k})
    \end{equation}
    where $\sign_{\symm_k}$ is the 1-dimensional sign representation of $\symm_k$ in degree 0 and the graded $(\symm_{n-k} \times \symm_{n-k})$-module $\RRR(\symm_{n-k})$ is regarded as a graded $\symm_{n-k}$ module via the diagonal embedding \[\Delta: \symm_{n-k} \to \symm_{n-k} \times \symm_{n-k}\] given by $\Delta(w) = (w,w).$
\end{lemma}

\begin{proof}
    The induction $\Ind_{\symm_{n-k} \times \symm_k}^{\symm_n} (\RRR(\symm_{n-k}) \otimes_\F \sign_{\symm_k})$ is given by
    \begin{align}
        \Ind_{\symm_{n-k} \times \symm_k}^{\symm_n} (\RRR(\symm_{n-k}) \otimes_\F \sign_{\symm_k}) &= \F[\symm_n] \otimes_{\F[\symm_{n-k} \times \symm_k]} ( \RRR(\symm_{n-k}) \otimes_\F \sign_{\symm_k}) \\
        &= \bigoplus_{|K|=k} \RRR(\symm_{n-k})^{(K)} \cdot \ee_K
    \end{align}
    where for $K \subseteq [n]$ with $|K| = k$, we have $w \cdot \ee_K = \sign(w) \cdot \ee_K$ for all $w \in \symm_K$.
\end{proof}

 The module structure of $\RRR(\DDD_n)$ will be an alternating graded sum of the $V_k$. To set up the proper inductive structure, we generalize \eqref{eq:v-definition} slightly. Given $T \subseteq [n]$, write
\begin{equation}
    \label{eq:definition-of-a}
    A_{n,T} := \RRR(\symm_n(\RRRR_T)) \quad \text{where } \RRRR_T = \{ (t,t) \,:\, t \in T \}.
\end{equation}
 In particular, we have
\begin{equation}
    A_{n,[n]} = \RRR(\DDD_n).
\end{equation}We set
\begin{equation}
    \label{eq:definition-of-c}
     C_{n,k}(T) := \bigoplus_{\substack{K \subseteq T \\ |K| = k}} \RRR(\symm_{n-k})^{(K)} \cdot \ee_K.
\end{equation}
In the case $T = [n]$ we have
\begin{equation}
\label{eq:c-and-v}
    C_{n,k}([n]) = V_k.
\end{equation}
The $A_{n,T}$ are related by the following short exact sequences.

\begin{lemma}
    \label{lem:a-ses}
    Let $n \geq 2$,
    let $T \subseteq [n]$ and let $t_0 \in T$. Regard $A_{n-1,T-t_0}$ as a quotient of \[S^{(t_0)} = \F[x_{i,j} \,:\, 1 \leq i,j \leq n, \, \, i,j \neq t_0]\] in the natural way. We have a short exact sequence
    \begin{equation}
        0 \to A_{n-1,T-t_0} \xrightarrow{ \, \, \varphi \, \, } A_{n,T-t_0} \xrightarrow{\, \, \pi \, \, } A_{n,T} \to 0
    \end{equation}
    where $\varphi$ is induced by multiplication by $x_{t_0,t_0}$ and $\pi$ is the canonical projection. 
\end{lemma}

At $n=1$, the sequence in Lemma~\ref{lem:a-ses} becomes 
 \begin{equation*}
        \begin{tikzpicture}[scale=0.8]
            \node(10) at (0,2) {$0$};
            \node(11) at (4,2) {$A_{0,\varnothing}$};
            \node(12) at (8,2) {$A_{1,\varnothing}$};
            \node(13) at (12,2) {$A_{1,\{1\}}$};
            \node(14) at (16,2) {$0$};

            \node(00) at (0,0) {$0$};
            \node(01) at (4,0) {$\F$};
            \node(02) at (8,0) {$\F$};
            \node(03) at (12,0) {$0$};
            \node(04) at (16,0) {$0$};

            \draw[->] (00) -- (01);
            \draw[->] (01) -- node[above] {$0$} (02);
            \draw[->] (02) -- (03);
            \draw[->] (10) -- (11);
            \draw[->] (13) -- (14);
            \draw[->] (03) -- (04);
            \draw[->] (11) -- node[above] {$\varphi$} (12);
            \draw[->] (12) -- node[above] {$\pi$} (13);
            \draw[-,double distance=1.5pt] (01) -- (11);
            \draw[-,double distance=1.5pt] (02) -- (12);
            \draw[-,double distance=1.5pt] (03) -- (13);
        \end{tikzpicture}
    \end{equation*}
which is not exact.

\begin{proof}
    This is the short exact sequence of Theorem~\ref{thm:ses-presentation} (1) where $\RRRR = \RRRR_T = \{ (t,t) \,:\, t \in T \}$ and $(i_0,j_0) = (t_0,t_0)$.
\end{proof}

\subsection{Mapping cones and resolutions} 
We show that the $C_{n,k}(T)$ assemble into exact sequences which give resolutions of the $A_{n,T}$.  These sequences take a slightly different form depending on whether $T$ is a proper subset of $[n]$. We handle the case of proper subsets $T \subsetneq [n]$ first.

\begin{lemma}
    \label{lem:T-less-than}
    Suppose $T \subsetneq [n]$ is a {\bf proper} subset of $[n]$. We have an exact sequence of $\F$-vector spaces
    \[
    0 \to C_{n,|T|}(T) \xrightarrow{\, \partial_{|T|} \, } \cdots \to C_{n,1}(T) \xrightarrow{\,\partial_1 \,} C_{n,0}(T) \xrightarrow{\, \partial_0 \,} A_{n,T} \to 0
    \]
    where the differentials $\partial_k: C_{n,k}(T) \to C_{n,k-1}(T)$ are given by
    \[
    \partial_k: \sum_{\substack{K \subseteq T \\ |K| = k}} f_K \cdot \ee_K \mapsto \sum_{s=1}^k (-1)^{s-1} x_{i_s,i_s} \cdot f_K \cdot \ee_{K-i_s} 
    \]
    where $K = \{i_1 < \cdots < i_k\} \subseteq T$ and $k > 0$. The final differential $\partial_0: C_{n,0}(T) \to A_{n,T}$ is the canonical projection 
    \[
    \partial_0 : \RRR(\symm_n)^{(\varnothing)} \cdot \ee_\varnothing = \RRR(\symm_n) \twoheadrightarrow A_{n,T}.
    \]
\end{lemma}

Observe that the maps $\partial_k$ are homogeneous degree 1 for $k > 0$ and that $\partial_0$ is homogeneous of degree 0.
 We will also use the mapping cone result of Theorem~\ref{thm:exact-cone} together with the short exact sequence Lemma~\ref{lem:a-ses} in the following proof.

\begin{proof}
    We use upward induction on $|T|$ and downward induction on $n$. If $T = \varnothing$ is empty, we have $A_{n,\varnothing} = \RRR(\symm_n)$ and the sequence in question is simply 
    \begin{equation*}
        \begin{tikzpicture}[scale=0.8]
            \node(10) at (0,2) {$0$};
            \node(11) at (4,2) {$C_{n,0}(\varnothing)$};
            \node(12) at (8,2) {$A_{n,\varnothing}$};
            \node(13) at (12,2) {$0$};

            \node(00) at (0,0) {$0$};
            \node(01) at (4,0) {$\RRR(\symm_n)$};
            \node(02) at (8,0) {$\RRR(\symm_n)$};
            \node(03) at (12,0) {$0$};

            \draw[->] (00) -- (01);
            \draw[->] (01) -- node[above] {$\mathrm{id.}$} (02);
            \draw[->] (02) -- (03);
            \draw[->] (10) -- (11);
            \draw[->] (11) -- node[above] {$\partial_0$} (12);
            \draw[->] (12) -- (13);
            \draw[-,double distance=1.5pt] (01) -- (11);
            \draw[-,double distance=1.5pt] (02) -- (12);
        \end{tikzpicture}
    \end{equation*}
    which is certainly exact.
    If $n = 1$, since $T$ is a proper subset we must have $T = \varnothing$ and we are reduced to the same exact sequence. 

    Suppose $T$ is a proper and nonempty subset of $[n]$. In particular, we have $n >1$. Let $t_0 = \max(T)$ be the largest element of $T$. Lemma~\ref{lem:a-ses} gives a short exact sequence
    \begin{equation}
    \label{eq:A-ses}
    0 \to A_{n-1, T- t_0} \to A_{n,T-t_0} \to A_{n,T} \to 0
    \end{equation}
    where $A_{n-1,T-i_0}$ is regarded in the $(n-1) \times (n-1)$ matrix of variables 
    \begin{equation}
    \label{eq:bar-x-set-t0}
    \overline{\xx} := (x_{i,j} \,:\, 1 \leq i, j \leq n, \, \, i,j \neq t_0 ).
    \end{equation}
    The first map $A_{n-1, T- t_0} \to A_{n,T-t_0}$ is induced by multiplication by $x_{t_0,t_0}$. Since $T - t_0$ is a proper subset of $[n] - t_0$, by induction we have exact sequences of $\F$-vector spaces
    \begin{equation}
    \label{eq:top-resolution-1}
    0 \to C_{n-1,|T|-1}(T - t_0) \to \cdots \to C_{n-1,1}(T - t_0) \to C_{n-1,0}(T - t_0) \to A_{n-1,T -t_0} \to 0
    \end{equation}
    and
    \begin{equation}
    \label{eq:bottom-resolution-1}
     0 \to C_{n,|T|-1}(T - t_0) \to \cdots \to C_{n,1}(T - t_0) \to C_{n,0}(T - t_0) \to A_{n,T -t_0} \to 0
    \end{equation}
    with differentials as described in the statement of the lemma.
    In the top exact sequence \eqref{eq:top-resolution-1} one has
    \begin{equation}
    C_{n-1,k}(T-t_0) = \bigoplus_{\substack{K \subseteq T - t_0 \\ |K| =k}} \RRR(\symm_{n-k-1})^{(K)} \cdot \ee_K \quad 
    \text{with $\RRR(\symm_{n-k-1})^{(K)}$ in the variable set $\bar{\xx}$.}
    \end{equation}

    Applying Lemmas~\ref{lem:K-maps} and \ref{lem:a-ses}  gives  a  chain map $f_\bullet$ from the top complex \eqref{eq:top-resolution-1} to the bottom complex \eqref{eq:bottom-resolution-1} whose components
    \begin{align*}
         C_{n-1,k}(T-t_0) &\to C_{n,k}(T-t_0)   &  A_{n-1,T-t_0} &\to A_{n,T-t_0} \\
         f \cdot \ee_K &\mapsto x_{t_0,t_0} \cdot f \cdot \ee_K & 
         f &\mapsto x_{t_0,t_0} \cdot f
    \end{align*}
    are induced by multiplication by $x_{t_0,t_0}$. Consider the mapping cone $\Cone(f_\bullet)$ of $f_\bullet$ with chain groups
    \begin{equation}
    \Cone(f_\bullet)_k = C_{n,k}(T-t_0) \oplus C_{n-1,k-1}(T-t_0) \cong_\F C_{n,k}(T)
    \end{equation}
    where the vector space isomorphism $C_{n,k}(T-t_0) \oplus C_{n-1,k-1}(T-t_0) \cong C_{n,k}(T)$ identifies 
    \begin{equation}
    \label{eq:signed-identification}
    (p \cdot \ee_K, q \cdot \ee_R) \leftrightarrow  -p \cdot \ee_K -  (-1)^k \cdot q \cdot \ee_R \wedge \ee_{t_0}
    \end{equation}
    for all size $k$ subsets $K \subseteq T - t_0$, all size $k-1$ subsets $R \subseteq T - t_0$, all $p \in \RRR(\symm_{n-k})^{(K)}$, and all $q \in \RRR(\symm_{n-k-1})^{(R)}$. The unusual signs in \eqref{eq:signed-identification} will be important momentarily.

    Theorem~\ref{thm:exact-cone} and the short exact sequence \eqref{eq:A-ses} give an exact sequence of $\F$-vector spaces
    \begin{equation}
    \label{eq:proper-cone-resolution}
        (\Cone(f_\bullet) \to A_{n,T} \to 0) = (0 \to C_{n,|T|}(T) \to \cdots \to C_{n,1}(T) \to C_{n,0}(T) \to A_{n,T} \to 0).
    \end{equation}
    Let us examine the differentials of \eqref{eq:proper-cone-resolution}. Suppose first that $k\geq1$.
    Thanks to our strategic choice of signs in \eqref{eq:signed-identification}, for $R = \{i_1 < \cdots < i_{k-1}\} \subseteq T - t_0$, the differential $d_k$ of $\Cone(f_\bullet)$ satisfies 
    \[
    d_k(0,q \cdot \ee_R) = \left(q \cdot x_{t_0,t_0} \cdot \ee_R, - \sum_{s=1}^{k-1} (-1)^{s-1} q \cdot x_{i_s,i_s} \cdot \ee_{R - i_s} \right)
    \]
    which translates to
    \[
        (-1)^k d_k(q \cdot \ee_R \wedge \ee_{t_0}) = - (q \cdot x_{t_0,t_0}) \cdot \ee_R + (-1)^k \sum_{s=1}^{k-1} (-1)^{s-1} (q \cdot x_{i_s,i_s}) \cdot \ee_{R - i_s} \wedge \ee_{t_0}
    \]
    or
    \[
    d_k(q \cdot \ee_R \wedge \ee_{t_0}) = (-1)^{k-1} q \cdot x_{t_0,t_0} \cdot \ee_R + \sum_{s=1}^{k-1} (-1)^{s-1} (p \cdot x_{i_s,i_s}) \cdot \ee_R \wedge \ee_{t_0}.
    \]
    This coincides with the differential $\partial_k: C_{n,k}(T) \to C_{n,k-1}(T)$ in the lemma since $t_0 = \max(T)$. Similarly, for $K = \{i_1 < \cdots < i_k\} \subseteq T - t_0$ one has 
    \[
    d_k(p \cdot \ee_K,0) = \left(
        \sum_{s=1}^{k} (-1)^{s-1} (p \cdot x_{i_s,i_s}) \cdot \ee_{K-i_s},0
    \right)
    \]
    which translates to
    \[
    d_k(-p \cdot \ee_K) = - \sum_{s=1}^k (-1)^{s-1} ( p \cdot x_{i_s,i_s}) \cdot \ee_{K - i_s}.
    \]
    This again agrees with the differential $\partial_k: C_{n,k}(T) \to C_{n,k-1}(T)$ in the lemma. At $k = 0$, the identification $d_0 = \partial_0$ of differentials is immediate from Theorem~\ref{thm:exact-cone} and the proof of the lemma is complete.
\end{proof}

We would like an exact sequence resolving $A_{n,[n]} = \RRR(\symm_n)$. Unfortunately, the proof of the base case in Lemma~\ref{lem:T-less-than} breaks down when $T = [n]$ and the conclusion of Lemma~\ref{lem:T-less-than} fails when $T = [n]$. In order to adjust Lemma~\ref{lem:T-less-than} to apply to the case $T = [n]$, we will need a simple result on the ring $\RRR(\symm_2)$.

\begin{lemma}
    \label{lem:s2}
    Inside the ring $\RRR(\symm_2)$ we have $x_{1,1} = x_{2,2}$.
\end{lemma}

\begin{proof}
    This follows from the relations $x_{1,1} + x_{1,2} = 0$ and $x_{1,2} + x_{2,2} = 0$.
\end{proof}

The $T = [n]$ version of Lemma~\ref{lem:T-less-than} reads as follows.

\begin{lemma}
    \label{lem:T-equals-n}
    For any $n \geq 0$ we have an exact sequence of $\F$-vector spaces 
    \[
    0 \to C_{n,n}([n]) \xrightarrow{\, \partial_n \,} \cdots \xrightarrow{\,\partial_2\,} C_{n,1}([n]) \xrightarrow{\, \partial_1 \,} C_{n,0}([n]) \xrightarrow{\, \partial_0 \, } A_{n,[n]} \to 0
    \]
    where $\dots$ 
    \begin{itemize}
    \item for $0 < k < n$, the differential $\partial_k: C_{n,k}([n]) \to C_{n,k-1}([n])$ is the degree 1 map defined in Lemma~\ref{lem:T-less-than},
    \item the differential $\partial_0: C_{n,0}([n]) \to A_{n,[n]}$ the degree 0 map defined in Lemma~\ref{lem:T-less-than}, and
    \item  the top differential $\partial_n: C_{n,n}([n]) \to C_{n,n-1}([n])$ is the degree 0 map
    \[
    \partial_n: f \cdot \ee_1 \wedge \cdots \wedge \ee_n \mapsto \sum_{p=1}^n  (-1)^{p-1} f \cdot  \ee_1 \wedge \cdots \wedge \widehat{\ee_p} \wedge \cdots \wedge \ee_n
    \]
    for all $f \in \RRR(\symm_0)^{[n]} \cong \F$.
    \end{itemize}
\end{lemma}

The exact sequence of Lemma~\ref{lem:T-equals-n} may be rewritten as
\begin{equation}
    \label{eq:V-resolution}
     0 \to V_n \xrightarrow{\, \partial_n \,} \cdots \xrightarrow{\,\partial_2\,} V_1 \xrightarrow{\, \partial_1 \,} V_0 \xrightarrow{\, \partial_0 \, } \RRR(\DDD_n) \to 0 
\end{equation}
using the identifications $C_{n,k}([n]) = V_k$ and $A_{n,[n]} = \RRR(\DDD_n)$. Lemma~\ref{lem:T-equals-n} therefore gives the graded character of $\RRR(\DDD_n)$ as a graded alternating sum of the graded characters of the $V_k$.

\begin{proof}
    We induct on $n$, treating two base cases $n=0$ and $n=1$ separately. When $n=0$, the sequence in question is 
     \begin{equation*}
        \begin{tikzpicture}[scale=0.8]
            \node(10) at (0,2) {$0$};
            \node(11) at (4,2) {$C_{0,0}(\varnothing)$};
            \node(12) at (8,2) {$A_{0,\varnothing}$};
            \node(13) at (12,2) {$0$};

            \node(00) at (0,0) {$0$};
            \node(01) at (4,0) {$\F$};
            \node(02) at (8,0) {$\F$};
            \node(03) at (12,0) {$0$};

            \draw[->] (00) -- (01);
            \draw[->] (01) -- node[above] {$\mathrm{id.}$} (02);
            \draw[->] (02) -- (03);
            \draw[->] (10) -- (11);
            \draw[->] (11) -- node[above] {$\partial_0$} (12);
            \draw[->] (12) -- (13);
            \draw[-,double distance=1.5pt] (01) -- (11);
            \draw[-,double distance=1.5pt] (02) -- (12);
        \end{tikzpicture}
    \end{equation*}
    which is exact. When $n=1$, the sequence is 
    \begin{equation*}
        \begin{tikzpicture}[scale=0.8]
            \node(10) at (0,2) {$0$};
            \node(11) at (4,2) {$C_{1,1}(\{1\})$};
            \node(12) at (8,2) {$C_{1,0}(\{1\})$};
            \node(13) at (12,2) {$A_{1,\{1\}}$};
            \node(14) at (16,2) {$0$};

            \node(00) at (0,0) {$0$};
            \node(01) at (4,0) {$\F$};
            \node(02) at (8,0) {$\F$};
            \node(03) at (12,0) {$0$};
            \node(04) at (16,0) {$0$};

            \draw[->] (00) -- (01);
            \draw[->] (01) -- node[above] {$\mathrm{id.}$} (02);
            \draw[->] (02) -- (03);
            \draw[->] (10) -- (11);
            \draw[->] (11) -- node[above] {$\partial_1$} (12);
            \draw[->] (12) -- node[above] {$\partial_0$} (13);
            \draw[->] (13) -- (14);
            \draw[->] (03) -- (04);
            \draw[-,double distance=1.5pt] (01) -- (11);
            \draw[-,double distance=1.5pt] (02) -- (12);
            \draw[-,double distance=1.5pt] (03) -- (13);
        \end{tikzpicture}
    \end{equation*}
    which is also exact. 

    Suppose $n > 1$. Lemma~\ref{lem:a-ses} yields a short exact sequence
    \begin{equation}
    \label{eq:a-top-ses}
        0 \to A_{n-1,[n-1]} \to A_{n,[n-1]} \to A_{n,[n]} \to 0
    \end{equation}
    where the map $A_{n-1,[n-1]} \to A_{n,[n-1]}$ is induced by multiplication by $x_{n,n}$. We inductively have an exact sequence
    \begin{equation}
    \label{eq:top-resolution-2}
    0 \to C_{n-1,n-1}([n-1]) \to \cdots \to C_{n-1,1}([n-1]) \to C_{n-1,0}([n-1]) \to A_{n-1,[n-1]} \to 0
    \end{equation}
    resolving $A_{n-1,[n-1]}$.
    Since $[n-1] \subsetneq [n]$, Lemma~\ref{lem:T-less-than} gives an exact sequence
    \begin{equation}
    \label{eq:bottom-resolution-2}
    0 \to C_{n,n-1}([n-1]) \to \cdots \to C_{n,1}([n-1]) \to C_{n,0}([n-1]) \to A_{n,[n-1]} \to 0
    \end{equation}
    resolving $A_{n,[n-1]}$.

    We claim that the degree 1 maps coming from Lemma~\ref{lem:K-maps}
   \begin{align*}
         C_{n-1,k}([n-1]) &\to C_{n,k}([n-1]) & &(0 \leq k < n-1) \\
         f \cdot \ee_K &\mapsto x_{n,n} \cdot f  \cdot \ee_K & &(K \subseteq [n-1] \text{ and } |K|=k)
    \end{align*}
    induced by multiplication by $x_{n,n}$ together the degree 1 map in lowest homological degree
    \begin{align*}
        A_{n-1,[n-1]} &\to A_{n,[n-1]} \\
        f &\mapsto x_{n,n} \cdot f
    \end{align*}
    and the degree 0 map in highest homological degree
    \begin{align*}
    C_{n-1,n-1}([n-1]) &\to C_{n,n-1}([n-1]) \\
    f \cdot \ee_{[n-1]} &\mapsto f \cdot \ee_{[n-1]}
    \end{align*}
    assemble to give a chain map $f_\bullet$ from the top complex \eqref{eq:top-resolution-2} to the bottom complex \eqref{eq:bottom-resolution-2}. The commutativity condition defining chain maps is immediate except in top homological degree where one has the square of maps
    \begin{equation}
        \label{eq:want-to-commute}
        \begin{tikzpicture}[scale = 0.8]
            \node(01) at (0,2) {$C_{n-1,n-1}([n-1])$};
            \node(11) at (6,2) {$C_{n-1,n-2}([n-1])$};
            \node(m1) at (-3,2) {$\ee_{[n-1]}$};
            \node(m0) at (-3,0) {$\ee_{[n-1]}$};
            \node(p1) at (9,2) {$\ee_{K}$};
            \node(p0) at (9,0) {$x_{n,n} \cdot \ee_{K}$};

            \node(00) at (0,0) {$C_{n,n-1}([n-1])$};
            \node(10) at (6,0) {$C_{n,n-2}([n-1])$};

            \draw[->] (01) -- (11);
            \draw[->] (00) -- (10);
            \draw[->] (01) -- (00);
            \draw[->] (11) -- (10);
            \draw[|->] (m1) -- (m0);
            \draw[|->] (p1) -- (p0);
        \end{tikzpicture}
    \end{equation}
    where $K \subseteq [n-1]$ satisfies $|K| = n-2$. Since $\RRR(\symm_0) = \RRR(\symm_1) = \F$, the square \eqref{eq:want-to-commute} is equivalent to the square
    \begin{equation}
        \label{eq:want-to-commute-2}
        \begin{tikzpicture}[scale = 0.8]
            \node(01) at (0,2) {$\F \cdot \ee_{[n-1]}$};
            \node(11) at (6,2) {$\bigoplus_{i=1}^{n-1} \F \cdot \ee_{[n-1] - i}$};
            \node(m1) at (-2,2) {$\ee_{[n-1]}$};
            \node(m0) at (-2,0) {$\ee_{[n-1]}$};
            \node(p1) at (11,2) {$\ee_{[n-1]-\{i\}}$};
            \node(p0) at (11,0) {$x_{n,n} \cdot \ee_{[n] - \{i,n\}}$};

            \node(a0) at (0,3) {$\ee_{[n-1]}$};
            \node(a1) at (6,3) {$\sum_{i=1}^{n-1}(-1)^i \cdot  \ee_{[n-1]-\{i\}}$};
            \node(b0) at (0,-1) {$\ee_{[n-1]}$};
            \node(b1) at (6,-1) {$\sum_{i=1}^{n-1}(-1)^i \cdot x_{i,i} \cdot  \ee_{[n]-\{i,n\}}$};

            \node(00) at (0,0) {$\F \cdot \ee_{[n-1]}$};
            \node(10) at (6,0) {$\bigoplus_{i=1}^{n-1} \RRR(\symm_2)^{[n] - \{i,n\}} \cdot \ee_{[n]-\{i,n\}}$};

            \draw[->] (01) -- (11);
            \draw[->] (00) -- (10);
            \draw[->] (01) -- (00);
            \draw[->] (11) -- (10);
            \draw[|->] (m1) -- (m0);
            \draw[|->] (p1) -- (p0);
            \draw[|->] (a0) -- (a1);
            \draw[|->] (b0) -- (b1);
        \end{tikzpicture}
    \end{equation}
    so commutativity of \eqref{eq:want-to-commute-2} (and \eqref{eq:want-to-commute}) is equivalent to the equation
    \begin{equation}
    \label{eq:2-syzygy}
        \sum_{i=1}^{n-1} (-1)^i \cdot x_{i,i} \cdot \ee_{[n] -\{i,n\}} = \sum_{i=1}^{n-1} (-1)^i \cdot x_{n,n} \cdot \ee_{[n] - \{i,n\}}
    \end{equation}
    inside the ring 
    \begin{equation}
        C_{n,n-2}([n-1]) = \bigoplus_{i=1}^{n-1} \RRR(\symm_2)^{[n] - \{i,n\}} \cdot \ee_{[n]-\{i,n\}}.
    \end{equation}
    Since $\RRR(\symm_2)^{[n] - \{i,n\}}$ is a copy of $\RRR(\symm_2)$ over the variable matrix $\begin{pmatrix} x_{i,i} & x_{i,n} \\ x_{n,i} & x_{n,n} \end{pmatrix}$,
   Lemma~\ref{lem:s2} implies that $x_{i,i} = x_{n,n}$ in the ring $\RRR(\symm_2)^{[n]-\{i,n\}}$ and \eqref{eq:2-syzygy} follows. We conclude that the square \eqref{eq:want-to-commute} is commutative and that $f_\bullet$ is a chain map.

    The rest of the proof is similar to that of Lemma~\ref{lem:T-less-than}.
    Since $f_\bullet$ is a chain map, the short exact sequence \eqref{eq:a-top-ses} and Theorem~\ref{thm:exact-cone} apply to give a resolution $\Cone(f_\bullet) \to A_{n,[n]} \to 0$.
    The $k^{th}$ chain group of $\Cone(f_\bullet)$ is 
    \begin{equation}
    \label{eq:cone-equal-identification}
       \Cone(f_\bullet)_k =    C_{n,k}([n-1])  \oplus C_{n-1,k-1}([n-1]) \cong_\F  C_{n,k}([n])
    \end{equation}
    for $0 \leq k \leq n$ where $C_{n,n}([n-1]) = C_{n-1,-1}([n-1]) = 0$.
    As in the proof of Lemma~\ref{lem:T-less-than}, the vector space isomorphism
    in \eqref{eq:cone-equal-identification} identifies
    \begin{equation}
    \label{eq:cone-equal-signs}
    (p \cdot \ee_K, q \cdot \ee_R) \leftrightarrow - p \cdot \ee_K - (-1)^k  q \cdot \ee_R \wedge \ee_n
    \end{equation}
    for all $K \subseteq [n-1]$ with $|K|=k$, all $R \subseteq [n-1]$ with $|R| = k-1$, all $p \in \RRR(\symm_{n-k})^{(K)}$, and all $q \in \RRR(\symm_{n-k-1})^{(R)}$.
    Also as in the proof of Lemma~\ref{lem:T-less-than}, the signs in \eqref{eq:cone-equal-signs} are chosen so that the differentials $d_k$ of $\Cone(f_\bullet)$ align with the differentials $\partial_k$ in the statement of the lemma. The proof that $d_k = \partial_k$ under the identification \eqref{eq:cone-equal-signs} is the same as in Lemma~\ref{lem:T-less-than} for $k < n$. When $k = n$, the map $d_n: \Cone(f_\bullet)_n \to \Cone(f_\bullet)_{n-1}$ translates to a map
    \begin{equation}
        d_n: 0 \oplus C_{n-1,n-1}([n-1]) \to C_{n,n-1}([n-1]) \oplus C_{n-1,n-2}([n-1])
    \end{equation}
    given by
    \begin{equation}
        d_n: (0, \ee_{[n-1]}) \mapsto \left(\ee_{[n-1]}, \sum_{i=1}^{n-1} (-1)^i \cdot \ee_{[n-1] - i}  \right).
    \end{equation}
    Under the identification \eqref{eq:cone-equal-signs} this corresponds to the map $d_n: C_{n,n}([n]) \to C_{n,n-1}([n])$ given by 
    \begin{equation}
        d_n: (-1)^n \cdot \ee_{[n-1]} \wedge \ee_n \mapsto - \ee_{[n-1]} - (-1)^n \cdot  \sum_{i=1}^n (-1)^i \cdot  \ee_{[n-1] - i} \wedge \ee_n.
    \end{equation}
    From here it is not difficult to see that $d_n = \partial_n$ as maps $C_{n,n}([n]) \to C_{n,n-1}([n])$.
\end{proof}

\subsection{\texorpdfstring{$\symm_n$}--structure} 
The group $\symm_n$ acts on $\DDD_n$ by conjugation, viz. $w \cdot v := w v w^{-1}$ for $w \in \symm_n$ and $v \in \DDD_n$. Write $\F[\DDD_n]$ for the corresponding permutation $\symm_n$-module.

\begin{proposition}
    \label{prop:ungraded-structure}
    Let $\symm_n$ act on $\DDD_n \subseteq \F^{n \times n}$ by conjugation. For $n > 1$ we have 
    \[
    S/(I_n + (x_{1,1}, \dots, x_{n,n})) = \RRR(\DDD_n)  \cong \FFF[\DDD_n].
    \]
    as ungraded $\symm_n$-modules.
\end{proposition}

\begin{proof}
    The identification $S/(I_n + (x_{1,1}, \dots, x_{n,n}))  = \RRR(\DDD_n)$ is Theorem~\ref{thm:ses-presentation} (2) and the isomorphism $\RRR(\DDD_n) \cong \F[\DDD_n]$ is the standard orbit harmonics isomorphism \eqref{eq:orbit-harmonics-isomorphism}.
\end{proof}

D\'esarm\'enien--Wachs (up to sign twist) introduced \cite{DW} and Reiner--Webb  studied \cite{RW} the (nonobviously genuine) $\symm_n$-character 
\begin{equation}
\chi_n := \sum_{k=0}^n (-1)^{n-k} \cdot \Ind_{\symm_1^k \times \symm_{n-k}}^{\symm_n} \one
\end{equation} 
of degree $|\DDD_n|$.  The character $\chi_n$ is {\bf not} that of $\F[\DDD_n]$ under the conjugation action. For example, we have 
\[
\chi_3 = V^{(2,1)} \quad \quad \text{whereas} \quad \quad \F[\DDD_3] \cong V^{(3)} \oplus V^{(1,1,1)}.
\]
The authors do not know an interesting graded refinement of $\chi_n$. The above example shows such a refinement cannot be a graded quotient of the coordinate ring $\F[W]$ of some $\symm_n$-module $W$.

We are ready to derive  the graded $\symm_n$-character of $\RRR(\DDD_n)$. Recall that $*$ is the Kronecker product on the ring $\Lambda$ of symmetric functions.

\begin{theorem}
\label{thm:graded-module}
The graded Frobenius image of $\RRR(\DDD_n)$ is given by
    \[
    \grFrob(\RRR(\DDD_n);q) = (-1)^n \cdot q^{n-1} \cdot e_n + \sum_{k=0}^{n-1} (-q)^k \cdot e_k \cdot \left[ \sum_{\lambda \vdash n-k} q^{n-k-\lambda_1} \cdot (s_\lambda * s_\lambda) \right]
    \]
    where $*$ is the Kronecker product.
\end{theorem}

\begin{proof}
    Lemma~\ref{lem:T-equals-n} implies that we have an exact sequence of $\F$-vector spaces 
    \begin{equation}
        \label{eq:grfrob-1}
        0 \rightarrow V_n \xrightarrow{\, \partial_n \, } \cdots \xrightarrow{\, \partial_2 \, } V_1 \xrightarrow{\, \partial_1 \, } V_0 \xrightarrow{\, \partial_0 \, } \RRR(\DDD_n) \to 0
    \end{equation}
    where $\partial_n, \partial_0$ are homogeneous of degree 0 and $\partial_k$ is homogeneous of degree $1$ for $0 < k < n$. The explicit formula in Lemma~\ref{lem:T-equals-n} shows that $\partial_k$ is an $\symm_n$-module homomorphism for all $0 \leq k \leq n$. It follows that 
    \begin{equation}
        \grFrob(\RRR(\DDD_n);q) = (-1)^n \cdot q^{n-1} \cdot \grFrob(V_n;q) + \sum_{k=0}^{n-1} (-q)^k \cdot \grFrob(V_k;q).
    \end{equation}
    For $0 \leq k \leq n$ the graded Frobenius image of $V_k$ is given by
    \begin{align}
        \grFrob(V_k;q) &= e_k \cdot \grFrob(\RRR(\symm_{n-k});q) \\
        &= e_k \cdot \sum_{\lambda \vdash n-k} q^{n-k-\lambda_1} \cdot (s_\lambda * s_\lambda)
    \end{align}
    where the first equality applies Lemma~\ref{lem:v-induction} and the second equality follows from Theorem~\ref{thm:permutation-matrix}. Plugging this formula for $\grFrob(V_k;q)$ into \eqref{eq:grfrob-1} yields the result.
\end{proof}

Reiner and Webb gave \cite{RW} a topological interpretation of $\chi_n$. In particular, let $P_n$ be the poset on words $a_1 \dots a_k$ over the alphabet $[n]$ without repetition with partial order 
\begin{equation*}
    b_1 \dots b_r \leq a_1 \dots a_k \quad \Leftrightarrow \quad 
    \text{there exist $1 \leq i_1 < \cdots < i_r \leq k$ with $b_p = a_{i_p}$ for $p = 1,\dots,r.$}
\end{equation*}
Then $P_n$ is the face poset of a regular simplicial complex $K_n$ which carries an $\symm_n$-action and whose reduced homology is concentrated in top degree. Reiner and Webb showed \cite{RW} that $\chi_n$ is the character of $\widetilde{H}_{\mathrm{top}}(K_n,\CC)$. It may be interesting to give a topological interpretation of the resolution
\[
0 \to V_n \to V_{n-1} \to \cdots \to V_0 \to \RRR(\DDD_n) \to 0
\]
of $\RRR(\DDD_n)$. 

The authors do not know a manifestly Schur-positive formula for $\grFrob(\RRR(\DDD_n);q)$. Since the Schur-expansion of Kronecker squares $s_\lambda * s_\lambda$ is far from understood, it would probably be difficult to derive such a formula from Theorem~\ref{thm:graded-module} directly. The first few examples of $\grFrob(\RRR(\DDD_n);q)$ are as follows.

\begin{center}
\renewcommand{\arraystretch}{1.2}
\begin{tabular}{c|l}
$n$ & $\grFrob(\RRR(\DDD_n);q)$ \\ \hline
1 & $0$ \\

2 & $s_{2}$ \\

3 & $q\,s_{1,1,1}+s_{3}$ \\

4 &
$q\,s_{2,1,1}
 +(q^{2}+q)\,s_{2,2}
 +(q^{2}+1)\,s_{4}$ \\

5 &
$\begin{aligned}
&q^{2}s_{1,1,1,1,1}
 +q^{2}s_{2,1,1,1}
 +q^{2}s_{2,2,1}
 +(q^{3}+q^{2}+q)s_{3,1,1} \\
&\qquad
 +(q^{2}+q)s_{3,2}
 +q^{2}s_{4,1}
 +(q^{2}+1)s_{5}
\end{aligned}$
\end{tabular}
\end{center}

\section{Conclusion}
\label{sec:Conclusion}

Although we have calculated the Hilbert series and graded module structure of $\RRR(\DDD_n)$, the authors do not know an explicit basis of this vector space. Unlike the case of $\RRR(\symm_n)$ in \cite{RhoadesViennot}, the ring $\RRR(\DDD_n)$ does not seem to have nice standard monomial theory. We give a conjectural basis of $\RRR(\DDD_n)$ involving the {\em patience sorting} algorithm.

Let $w= [w(1), \dots, w(n)] \in \symm_n$. We build a sequence of piles by reading the letters of $w$ from left to right. Initially there are no piles. Suppose that at step $i$ we read the letter $w(i)=a$. If there is a pile whose current top element $c$ satisfies $c>a$, then place $a$ on top of the leftmost such pile. In this case we say that $a$ \emph{bumps} $c$, and we call the pair $(a,c)$ a \emph{bump}. 
If no such pile exists, then we start a new pile consisting only of $a$ on the right. We write $\Bump(w)$ for the set of bumps of $w$, set $\bump(w) := |\Bump(w)|$, and let  $\pile(w)$ be the number of piles at the end of the algorithm. Hammersley proved \cite{Hammersley} that 
\begin{equation}
    \pile(w) = \lis(w).
\end{equation}
For example, if $w = [4,6,2,5,7,1,3] \in \symm_7$ the patience sorting algorithm proceeds as
\begin{multline*}
\varnothing \quad \leadsto \quad 4 \quad \leadsto \quad 4 \mid 6 \quad \leadsto \quad  2 \, 4 \mid 6 \quad \leadsto \quad 2 \, 4 \mid 5 \, 6 \quad \\ 
\leadsto \quad 2 \, 4 \mid 5 \, 6 \mid 7 \quad \leadsto \quad 1 \, 2 \, 4 \mid 5 \, 6 \mid 7 \quad \leadsto \quad 1 \, 2 \,4 \mid 3 \, 5 \, 6 \mid 7
\end{multline*}
In this case we have $\Bump(w) = \{ (4,2), (6,5), (2,1), (5,3) \}$ so that $\bump(w) = 4$. We have $\pile(w) = 3$ and indeed $\lis(w) = 3$.

Let $\Psi: \symm_n \xrightarrow{\,\,\sim\,\,} \symm_n$ be the Foata transformation. For any $w \in \symm_n$, define the {\em bump monomial} $\bbbb(w) \in S$ by
\begin{equation}
    \bbbb(w) := \prod_{(a,c) \in \Bump(\Psi(w))} x_{w(a),c}.
\end{equation}
Since 
\begin{equation}\bump(w) + \pile(w) = n \quad \quad \text{for all } w \in \symm_n,
\end{equation} the monomial $\bbbb(w)$ has degree $n-\lis(\Psi(w))$. The following conjecture is therefore consistent with Theorem~\ref{thm:hilbert}.

\begin{conjecture}
    \label{conj:basis}
    The set $\{ \bbbb(w) \,:\, w \in \DDD_n \}$ of bump monomials indexed by derangements descends to a basis of $\RRR(\DDD_n)$.
\end{conjecture}

\section*{Acknowledgements}

The authors thank Vic Reiner for asking about the structure of $\RRR(\DDD_n)$. B. Rhoades was partially supported by NSF Grant DMS-2246846. B. Rhoades got the idea of using mapping cones after long conversations with ChatGPT. The authors used ChatGPT for assistance in revising the paper. The authors wrote the entire paper themselves and take full responsibility for its contents.


\begin{thebibliography}{99}



\bibitem{BDJ} J. Baik, P. Deift, and K. Johansson. On the Distribution of the Length of the Longest Increasing Subsequence of Random Permutations,
{\em J. Amer. Math. Soc.} {\bf 12 (4)} (1999), 1119--1178.


\bibitem{DW} J. D\'esarm\'enien and M. Wachs. Descentes des d\'erangements et mots circulaires. Sem. Lotharing. Combin. {\bf 19} (1988), 13--21.


\bibitem{Foata} D. Foata. On the Netto inversion number of a sequence. {\em Proc. Amer. Math. Soc.}, {\bf 19} (1968), 236--240.

\bibitem{GP} A. M. Garsia and C. Procesi. On certain graded $S_n$-modules and the $q$-Kostka polynomials. {\em Adv. Math.}, {\bf 94 (1)} (1992), 82--138.


\bibitem{Griffin} S. Griffin. Ordered set partitions, Garsia-Procesi modules, and rank varieties. {\em Trans. Amer. Math. Soc.}, {\bf 374 (4)} (2021), 2609--2660.



\bibitem{HRS} J. Haglund, B. Rhoades, and M. Shimozono. Ordered set partitions, generalized coinvariant algebras, and the Delta Conjecture. {\em Adv. Math.}, {\bf 329} (2018), 851--915.

\bibitem{Hammersley} J. Hammersley. A few seedlings of research. {\em Proc. Sixth Berkeley Symp. Math. Statist. and Probability.} Vol. 1. University of California Press. 345--394.


\bibitem{Kostant} B. Kostant. Lie group representations on polynomial rings. {\em Amer. J. Math.}, {\bf 85} (1963), 327--404.


\bibitem{Liu} M. J. Liu. Viennot shadows and graded module structure in colored permutation groups. {\em Comb. Theory}, {\bf 5 (2)} (2025), \#7.


\bibitem{LMRZ} J. Liu, Y. Ma, B. Rhoades, and H. Zhu. Involution matrix loci and orbit harmonics. {\em Math. Z.}, Vol. 310, Article Number 23 (2025).


\bibitem{LZ} J. Liu and H. Zhu. An extension of Viennot's shadow to rook placements via orbit harmonics. Preprint, 2025. {\tt arXiv:2510.23735}.

\bibitem{OR} J. Oh and B. Rhoades. Zigzags, contingency tables, and quotient rings. {\em J. London Math. Soc.}, Vol. 112, Issue 3 (2025),  e70344.



\bibitem{RRT} M. Reineke, B. Rhoades, and V. Tewari. Zonotopal algebras, orbit harmonics, and Donaldson-Thomas invariants of symmetric quivers. {\em Int. Math. Res. Notices}, Vol. 2023, No. 23, 20169--20210.


\bibitem{RR} V. Reiner and B. Rhoades. Harmonics and graded Ehrhart theory. To appear, {\em J. Comb. Algebra}, 2026. {\tt arXiv:2407.06511}.


\bibitem{RW} V. Reiner and P. Webb. The combinatorics of the bar resolution in group homology. {\em J. Pure Appl. Alg.}, {\bf 190} (2004), 291--327.

\bibitem{RhoadesViennot} B. Rhoades. Increasing subsequences, matrix loci, and Viennot shadows. {\em Forum Math. Sigma}, (2024).  Vol. 12:e97, 1-23.



\bibitem{Sagan} B. Sagan. {\em The Symmetric Group: Representations, Combinatorial Algorithms, and Symmetric Functions}, 2nd ed., Graduate Texts in Mathematics, Vol. 203, Springer, New York, 2001.


\bibitem{Schensted} C. Schensted. Longest Increasing and Decreasing Subsequences, {\em Canad. J. Math.} {\bf 13} (1961), 179--191.

\bibitem{Viennot} G. Viennot. Une forme g\'eom\'etrique de la correspondance de Robinson–Schensted, in Combinatoire et Repr\'esentation du Groupe Sym\'etrique, Lecture Notes in Mathematics 579, Springer, 1977.


\bibitem{Weibel} C. Weibel. {\em An Introduction to Homological Algebra}, Cambridge Studies in Advanced Mathematics 38, Cambridge University Press, 1994.


\bibitem{Zhu} H. Zhu. Rook placements and orbit harmonics. Preprint, 2025. {\tt arXiv:2510.25106}.




  
 \end{thebibliography}
\end{document}